\documentclass[11pt,a4paper]{amsart}

\usepackage{amsmath,amsthm}
\usepackage{mathrsfs}
\usepackage{graphicx}
\usepackage[latin1]{inputenc}
\usepackage[numbers]{natbib}
\usepackage{a4wide}
\usepackage{url}
\usepackage{stackrel}
\usepackage{txfonts}
\usepackage{hyperref} 
\usepackage[capitalise]{cleveref} 

\newcommand{\1}{\mathbf{1}}
\newcommand{\eps}{\epsilonup}
\newcommand{\tn}[1]{\textnormal{#1}}

\newcommand{\mrm}[1]{\mathrm{#1}}
\newcommand{\mc}[1]{\mathcal{#1}}

\newcommand{\wh}{\widehat}
\newcommand{\wt}{\widetilde}
\newcommand{\N}{\mathbb{N}}

\newcommand{\R}{\mathbb{R}}

\newcommand{\diag}{\mathrm{diag}}
\newcommand{\T}{\mathsf{T}}

\newcommand{\eqd}{\stackrel{D}{=}}
\newcommand{\var}{\mathrm{Var}}
\newcommand{\cov}{\mathrm{Cov}}
\newcommand{\sto}[2]{\stackrel[#2]{#1}{\longrightarrow}}
\newcommand{\ston}[1]{\stackrel[n\to\infty]{#1}{\longrightarrow}}
\renewcommand{\sp}[2]{\left\langle #1,#2 \right\rangle}
\newcommand{\norm}[1]{\left\|#1\right\|}
\newcommand{\bpm}{\begin{pmatrix}}
\newcommand{\epm}{\end{pmatrix}}%

\newtheorem{theorem}{Theorem}
\newtheorem{proposition}{Proposition}[section]
\newtheorem{lemma}{Lemma}[section]

\newtheorem{corollary}{Corollary}

\theoremstyle{definition}

\begin{document}

\title{Limit Theory for the largest eigenvalues of sample covariance matrices with heavy-tails}
\author[R.A. Davis]{Richard A. Davis}
\address{Department of Statistics\\ Columbia University \\ 1255 Amsterdam Avenue \\ New York, New York 10027 \\ USA}
\email{rdavis@stat.columbia.edu}
\author[O. Pfaffel]{Oliver Pfaffel}
\address{Fakultät für Mathematik \& TUM-IAS \\ Technische Universität München \\ Boltzmannstraße 3 \\ 85748 Garching b. München \\ Germany}
\email{o.pfaffel@gmx.de}
\author[R. Stelzer]{Robert Stelzer}
\address{Institute of Mathematical Finance \\ Ulm University\\ Helmholtzstraße 18 \\ 89081 Ulm \\ Germany}
\email{robert.stelzer@uni-ulm.de}

%
%
%
%
%

\begin{abstract}
We study the joint limit distribution of the $k$ largest eigenvalues of a $p\times p$ sample covariance matrix $XX^\T$ based on a large $p\times n$ matrix $X$. The rows of $X$ are given by independent copies of a linear process, $X_{it}=\sum_j c_j Z_{i,t-j}$, with regularly varying noise $(Z_{it})$ with tail index $\alpha\in(0,4)$. It is shown that a point process based on the eigenvalues of $XX^\T$ converges, as $n\to\infty$ and $p\to\infty$ at a suitable rate, in distribution to a Poisson point process with an intensity measure depending on $\alpha$ and $\sum c_j^2$. This result is extended to random coefficient models where the coefficients of the linear processes $(X_{it})$ are given by $c_j(\theta_i)$, for some ergodic sequence $(\theta_i)$, and thus vary in each row of $X$.
As a by-product of our techniques we obtain a proof of the corresponding result for matrices with iid entries in cases where $p/n$ goes to zero or infinity and $\alpha\in(0,2)$.
\end{abstract}

\keywords{Random Matrix Theory, heavy-tailed distribution, random matrix with dependent entries, largest singular value, sample covariance matrix,
largest eigenvalue, linear process, random coefficient model}

\subjclass[2010]{60B20, 62G32, 60G55, 62H25}

\maketitle

\section{Introduction}

Recently there has been increasing interest in studying \emph{large dimensional data sets} that arise in finance, wireless communications, genetics and other fields. Patterns in these data can often be summarized by the \emph{sample covariance matrix}, as done in multivariate regression and dimension reduction via factor analysis. 
Therefore, our objective is to study the asymptotic behavior of the eigenvalues $\lambda_{(1)}\geq\ldots\geq\lambda_{(p)}$ of a $p\times p$ sample covariance matrix $XX^\T$, where the \emph{data matrix} $X$ is obtained from $n$ observations of a high-dimensional stochastic process with values in $\R^p$. Classical results in this direction often assume that the entries of $X$ are independent and identically distributed (iid) or satisfy some moment conditions. 
For example, the Four Moment Theorem of Tao and Vu \cite{Tao2012} shows that the asymptotic behaviour of the eigenvalues of $XX^\T$ is determined by the first four moments of the distribution of the iid matrix entries of $X$.
Our goal is to weaken the moment conditions by allowing for heavy-tails, and the assumption of independent entries by allowing for dependence within the rows and columns. Potential applications arise in portfolio management in finance, where observations typically have heavy-tails and dependence.

Assuming that the data comes from a multivariate normal distribution, one is able to compute the joint distribution of the eigenvalues $(\lambda_{(1)},\ldots,\lambda_{(p)})$, see \cite{James1964}. Under the additional assumption that the dimension $p$ is fixed while the sample size $n$ goes to infinity, Anderson \cite{Anderson1963} obtains a central limit like theorem for the largest eigenvalue. Clearly, it is not possible to derive the joint distribution in a general setting where the distribution of $X$ is not invariant with respect to orthogonal transformations. Furthermore, since in modern applications with large dimensional data sets, $p$ might be of similar or even larger order than $n$, it might be more suitable to assume that both $p$ and $n$ go to infinity, so Anderson's result may not be a good approximation in this setting. For example, considering a financial index like the S\&P 500, the number of stocks is $p=500$, whereas, if daily returns of the past 5 years are given, $n$ is only around 1300. In genetic studies, the number of investigated genes $p$ might easily exceed the number of participating individuals $n$ by several orders of magnitude. In this \emph{large $n$, large $p$} framework results differ dramatically from the corresponding \emph{fixed $p$, large $n$} results - with major consequences for the statistical analysis of large data sets \cite{Johnstone2001}.

Spectral properties of large dimensional random matrices is one of many topics that has become known under the banner \emph{Random Matrix Theory (RMT)}. 
The original motivation for RMT comes from mathematical physics \cite{Dyson1962}, \cite{Wigner1958}, where large random matrices serve as a finite-dimensional approximation of infinite-dimensional operators. Its importance for statistics comes from the fact that RMT may be used to correct traditional tests or estimators which fail in the `large $n$, large $p$' setting. For example, Bai et al. \cite{Bai2009} gives corrections on some likelihood ratio tests that fail even for moderate $p$ (around 20), and El Karoui \cite{ElKaroui2008} consistently estimates the spectrum of a large dimensional covariance matrix using RMT. Thus statistical considerations will be our motivation for a random matrix model with heavy-tailed and dependent entries.

Before describing our results, we will give a brief overview of some of the key results from RMT for real-valued sample covariance matrices $XX^\T$. A more detailed account on RMT can be found, for instance, in the textbooks \cite{Anderson2009}, \cite{Bai2010}, or \cite{Mehta2004}. Here $X$ is a real $p\times n$ random matrix, and $p$ and $n$ go to infinity simultaneously. Let us first assume that the entries of $X$ are iid with variance 1. Results on the \emph{global behavior of the eigenvalues} of $XX^\T$ mostly concern the \emph{spectral distribution}, that is the random probability measure of its eigenvalues $p^{-1} \sum_{i=1}^p \eps_{n^{-1}\lambda_{(i)}}$, where $\eps$ denotes the Dirac measure. The spectral distribution converges,  as $n,p\to\infty$ with $p/n\to\gamma\in(0,1]$, to a deterministic measure with density function
\[ \frac{1}{2\pi x\gamma}\sqrt{(x_{+}-x)(x-x_{-})} \1_{(x_{-},x_{+})}(x), \quad x_{\pm}\coloneqq (1\pm\sqrt{\gamma})^2, \]
where $\1$ denotes the indicator function. This is the so called \emph{Mar{\v{c}}enko--Pastur law} \cite{Marchenko1967}, \cite{Wachter1978}. 
One obtains a different result if $XX^\T$ is perturbed via an affine transformation \cite{Marchenko1967}, \cite{Pan2010}.
Based on these results, \cite{Pfaffel2011} treats the case where the rows of $X$ are given by independent copies of a linear process. Apart from a few special cases, the limiting spectral distribution is not known in a closed form if the entries of $X$ are not independent.

Although the eigenvalues of $XX^\T$ offer various interesting local properties to be studied, we will only focus on the joint asymptotic behavior of the $k$ largest eigenvalues $(\lambda_{(1)},\ldots,\lambda_{(k)})$, $k\in\N$. This is motivated from a statistical point of view since the variances of the first $k$ principal components are given by the $k$ largest eigenvalues of the covariance matrix. Geman \cite{Geman1980} shows, assuming that the entries of $X$ are iid and have finite fourth moments, that $n^{-1}\lambda_{(1)}$ converges to $x_{+}=(1+\sqrt{\gamma})^2$ almost surely if $p/n\to\gamma\in(0,\infty)$. Moreover, if the entries of $X$ are iid standard Gaussian, Johnstone \cite{Johnstone2001} shows that
\[ \frac{\sqrt{n}+\sqrt{p}}{\sqrt[3]{\frac{1}{\sqrt{n}}+\frac{1}{\sqrt{p}}}}\left(\frac{\lambda_{(1)}}{\left(\sqrt{n}+\sqrt{p}\right)^2}-1\right) \sto{D}{} \xi, \]
where $\xi$ follows the \emph{Tracy--Widom distribution} with $\beta=1$. Soshnikov \cite{Soshnikov2002} extends this to more general symmetric non-Gaussian distributions if the matrix $X$ is nearly square, and obtains a similar result for the joint convergence of the $k$ largest eigenvalues. The Tracy--Widom distribution first appeared as the limit of the largest eigenvalue of a Gaussian Wigner matrix \cite{Tracy1994}. Péché \cite{Peche2009} shows that the assumption of Gaussianity in Johnstone's result can be replaced by the assumption that the entries of $X$ have a symmetric distribution with sub-Gaussian tails, and she allows for $\gamma$ being zero or infinity.

There exist results on extreme eigenvalues of $XX^\T$ which include dependence within the rows or columns of $X$, but most of them are only valid if $X$ has complex-valued entries such that its real as well as its complex part have a non-zero variance.  A notable exception, where the real-valued case is considered, is \cite{Bloemendal2010}. They assume that the rows of $X$ are normally distributed with a covariance matrix which has exactly one eigenvalue not equal to one.

In contrast to the light tailed case described above, there exist only a handful of articles dealing with sample covariance matrices $XX^\T$ obtained from heavy-tailed observations. All these results only apply to matrices $X$ with iid entries. Belinschi et al. \cite{Belinschi2009} compute the limiting spectral distribution of sample covariance matrices based on observations with infinite variance. Regarding the $k$-largest eigenvalues, Soshnikov \cite{Soshnikov2006} gives the weak limit in case the underlying distribution of the matrix entries is Cauchy. Biroli et al. \cite{Biroli2007} argued, using heuristic arguments and numerical simulations, that Soshnikov's result extends to general distributions with regularly varying tails with index $0<\alpha<4$. A mathematically rigorous proof of this claim followed by Auffinger et al. \cite{Auffinger2009}.


We extend the previous results for $0<\alpha<4$ by allowing for dependent entries. More specifically, the rows of $X$ are given by independent copies of some linear process. Their respective coefficients can either all be equal (\cref{first result}) or, more generally, conditionally on a latent process, vary in each row (\cref{extension}). In the latter case the rows of $X$ are not necessarily independent. The limiting Poisson process of the eigenvalues of $XX^\T$ depends on the tail index $\alpha$ as well as the coefficients of the observed linear processes. As a by-product, we obtain an independent proof of Soshnikov's result for iid entries which also holds in cases where $\gamma\in\{0,\infty\}$.

The paper is organized as follows. The main results will be presented in Section 2 while the proofs will be given in Section 3. Results from the theory of point processes and regular variation are required through most of this paper.
A detailed account on both topics can be found in a number of texts. We mainly adopt the setting, including notation and terminology, of Resnick \cite{Resnick2008}.

\section{Main results on heavy-tailed random matrices with dependent entries}

\subsection{A first result on the largest eigenvalue}\label{first result}

Let $(Z_{it})_{i,t}$ be an array of iid random variables with marginal distribution that is regularly varying with tail index $\alpha>0$ and \emph{normalizing sequence} $a_n$, i.e.,
\begin{align}
\lim_{n\to\infty}nP(|Z_{it}|>a_nx)=x^{-\alpha}, \quad\tn{for each } x>0.
\label{Z reg var}
\end{align}
Equivalently, this means that $(|Z_{it}|)$ is in the maximum domain of attraction of a Fr\'echet distribution with parameter $\alpha>0$.
The sequence $a_n$ is then necessarily characterized by
\begin{align}
a_n=n^{1/\alpha} L(n), 
\label{a_n}
\end{align}
for some slowly varying function $L:\R_+\to\R_+$, i.e., a function with the property that, for each $x>0$, $\lim_{t\to\infty}{L(tx)/L(t)} = 1$.
In certain cases we also assume that $Z_{11}$ satisfies the \emph{tail balancing condition}, i.e., the existence of the limits
\begin{align}
\lim_{x\to\infty}\frac{P(Z_{11}>x)}{P(|Z_{11}|>x)}=q \quad\tn{and}\quad \lim_{x\to\infty}\frac{P(Z_{11}\leq -x)}{P(|Z_{11}|>x)}=1-q
\label{balancing}
\end{align}
for some $0\leq q\leq 1$. For each $p,n\in\N$, let $X=(X_{it})$
be the $p\times n$ data matrix, where, for each $i$,
\begin{align}
X_{it} = \sum_{j=-\infty}^\infty c_j Z_{i,t-j} 
\label{def Xit}
\end{align}
is a stationary linear times series. To guarantee that the series in \eqref{def Xit} converges almost surely, we assume that
\begin{align}
\sum_{j=-\infty}^\infty |c_j|^\delta<\infty \quad\tn{ for some } \delta<\min\{\alpha,1\}.
\label{summability c_j}
\end{align}
Thus in our model the rows of $X$ are given by iid copies of a linear process. We denote by $\lambda_{1},\ldots,\lambda_{p}\geq 0$ the eigenvalues of the $p\times p$ sample covariance matrix $XX^\T$. They are studied via the induced point process
\begin{align} 
N_n =  \sum_{i=1}^p \eps_{a_{np}^{-2}\left(\lambda_{i}-n\mu_{X,\alpha}\right)},
\label{pp eigenvalues}
\end{align}
where
\begin{align}\label{mu X alpha}
\mu_{X,\alpha} = \left\{\begin{array}{ll}  0 & \tn{ for } 0<\alpha<2,\\
 E\left(Z_{11}^2\1_{\{Z_{11}^2\leq a_{np}^2\}}\right) \sum_j c_j^2 &  \tn{ for } \alpha=2 \tn{ and } EZ_{11}^2=\infty,\\
 E\left(Z_{11}^2\right) \sum_j c_j^2 & \tn{ else.} \end{array}\right.
\end{align}
Since we are only interested in the largest eigenvalues, we consider $N_n$ as a point process on $(0,\infty)$ and only count eigenvalues $\lambda_i$ which are positive. Observe that the centralization term $n\mu_{X,\alpha}$ is equal to the mean of the diagonal elements of $XX^\T$ 
if the observations have a finite variance. In case the observations have an infinite variance, we do not have to center, except when $\alpha=2$ and $EZ_{11}^2=\infty$, where we use a truncated version of the mean. In the latter case $\mu_{X,\alpha}$ also depends on $p$ and $n$.

 We will always assume that $p=p_n$ is an integer-valued sequence in $n$ that goes to infinity as $n\to\infty$ in order to obtain results in the `large $n$, large $p$' setting. In the following we suppress the dependence of $p$ on $n$ so as to simplify the notation wherever this does not cause any ambiguity. In \cite{Auffinger2009,Soshnikov2006} the iid case is considered, i.e., $X_{it}=Z_{it}$, 
assuming that the condition \eqref{Z reg var} holds for $0<\alpha<4$. They show, if $p,n\to\infty$ with
\begin{align}
\lim_{n\to\infty}\frac{p_n}{n}=\gamma\in(0,\infty),
\label{p/n to gamma}<
\end{align}
that
\begin{align}
\sum_{i=1}^p \eps_{a_{np}^{-2}\lambda_{i}} \ston{D} N,
\end{align}
where $N$ is a Poisson process with intensity measure $\hat{\nu}((x,\infty])=x^{-\alpha/2}$. Our next theorem extends this result by considering the case where $X$ has dependent entries. 
More precisely, the rows of $X$ are given by independent copies of a linear process. It will turn out that the intensity measure of the limiting Poisson process depends on the sum of the squared coefficients of the underlying linear process.
In contrast to \cite{Auffinger2009}, we necessarily have to center the eigenvalues $\lambda_i$ by $n\mu_{X,\alpha}$ when $\alpha\geq 2$, since in that case we consider a regime where $p\approx n^\beta$ with $\beta<1$ instead of \eqref{p/n to gamma}.

\begin{theorem}\label{th1}
Define the matrix $X=(X_{it})$ as in equations \eqref{Z reg var}, \eqref{def Xit} and \eqref{summability c_j} 
with $\alpha\in(0,4)$. Suppose $p_n,n\to\infty$ such that
\begin{align}
\limsup_{n\to\infty} \frac{p_n}{n^\beta}<\infty
\label{beta con}
\end{align}
for some $\beta>0$ satisfying
\begin{enumerate}
\item $\beta<\infty$ if $0<\alpha\leq 1$, and 
\item $\beta<\max\left\{\frac{2-\alpha}{\alpha-1},\frac12\right\}$ if $1<\alpha<2$.
\item $\beta<\max\left\{\frac{1}{3},\frac{4-\alpha}{4(\alpha-1)}\right\}\quad$ if $\quad 2\leq\alpha< 3$, or
\item $\beta<\frac{4-\alpha}{3\alpha-4}\quad$ if $\quad 3\leq\alpha<4$.
\end{enumerate} 
Further assume, in case $\alpha\in(5/3,4)$, that $Z_{11}$ has mean zero and satisfies the tail balancing condition \eqref{balancing}.
Then the point process $N_n$, as defined in \eqref{pp eigenvalues}, converges in distribution to a Poisson point process $N$ with intensity measure $\nu$ which is given by \[ \nu((x,\infty])=x^{-\alpha/2} \left|\sum_{j=-\infty}^\infty c_j^2 \right|^{\alpha/2}, \quad x>0. \]
\end{theorem}

\cref{th1} weakens the assumption of independent entries made so far in the literature on heavy-tailed random matrices
at the expense of assumption \eqref{beta con}, which is more restrictive than the usual assumption \eqref{p/n to gamma} if $\alpha\in[1.5,4)$. However, if $\alpha\in(0,1.5)$, our assumption \eqref{beta con} is more general than \eqref{p/n to gamma}. This is important for statistical applications, because $p$ and $n$  are usually fixed and there is no functional relationship between the two of them. 

If we restrict ourselves to the iid case, then \cref{th1b} shows that the point process convergence result also holds in many cases where the limit $\gamma$ from condition \eqref{p/n to gamma} is zero or infinity, for example, by assuming that $p$ is regularly varying in $n$.

\begin{theorem}\label{th1b}
Assume that $X_{it}=Z_{it}$ and equation \eqref{Z reg var} is satisfied with $\alpha\in(0,2)$. Further, let either
\begin{enumerate}
 \item $p_n=n^\kappa l(n)$ for some $\kappa\in[0,\infty)$, where $l$ is a slowly varying function which converges to infinity if $\kappa=0$, and is bounded away from zero if $\kappa=1$, or
 \item $p_n\sim C\exp(cn^\kappa)$ for some $\kappa,c,C>0$.
\end{enumerate}
Then $N_n$ converges in distribution to a Poisson point process with intensity measure given by $\hat{\nu}((x,\infty])=x^{-\alpha/2}$.
\end{theorem}

It is well known \cite{Resnick2008} that a Poisson process has an explicit representation as a transformation of a homogeneous Poisson process. In our case, the limiting Poisson process $N$ with intensity measure $\nu$ from \cref{th1} can be written as
\begin{align}
N\eqd\sum_{i=1}^\infty \eps_{\Gamma_i^{-2/\alpha} \sum_{j=-\infty}^\infty c_j^2},
\label{N via Gamma}
\end{align}
where $\Gamma_i=\sum_{k=1}^{i}E_k$ is the successive sum of iid exponential random variables $E_k$ with mean one. The points of $N$ are labeled in decreasing order so that, by the continuous mapping theorem, we can easily deduce the weak limit of the (centered) $k$ largest eigenvalues of $XX^\T$.

\begin{corollary}
Denote by $\lambda_{(1)}\geq\ldots\geq\lambda_{(p)}$ the upper order statistics of the eigenvalues of $XX^\T-n\mu_{X,\alpha}I_p$. Under the assumptions of \cref{th1} we have, for each fixed integer $k\geq 1$, that the $k$-largest eigenvalues jointly converge,
\begin{align*}
a_{np}^{-2} \left(\lambda_{(1)},\ldots,\lambda_{(k)}\right) \ston{D} \left(\Gamma_1^{-2/\alpha},\ldots,\Gamma_k^{-2/\alpha}\right) \left( \sum_{j=-\infty}^\infty c_j^2\right).
\end{align*}
In particular, for each $x>0$, 
\[ P\left( \frac{\lambda_{(k)}}{a_{np}^2} \leq x \right) \ston{} P(N(x,\infty)\leq k-1)= e^{-x^{-\alpha/2}} \sum_{m=0}^{k-1}{ \frac{x^{-m\alpha/2}}{m!} \left(\sum_{j}{c_j^2}\right)^{m\alpha/2}}. \]
This implies for the largest eigenvalue $\lambda_{(1)}$ of $XX^\T-n\mu_{X,\alpha}I_p$ that
\begin{align*}
\frac{\lambda_{(1)}}{a_{np}^2 \sum_j c_j^2} \ston{D} V,
\end{align*}
where $V$ has a Fr\'echet distribution with parameter $\alpha/2$, i.e., $P(V\leq x)=e^{-x^{-\alpha/2}}$.\\
\label{corollary after th1}\end{corollary}



In a nutshell, the results in this section give the asymptotic behavior of the $k$ largest eigenvalues of a sample covariance matrix $XX^\T$ when the rows of $X$ are given by iid copies of some linear process with infinite variance. 
Our results will be generalized further in \cref{extension}, where, conditionally on a latent process, the rows of $X$ will be independent but not identically distributed.

\subsection{Examples and discussion}

\cref{th1} holds for any linear process which has regularly varying noise with infinite variance as long as condition \eqref{summability c_j} is satisfied. 
Since the coefficients of a causal ARMA process decay exponentially, \eqref{summability c_j} is trivially satisfied in this case.
As two simple examples, consider an MA(1) process $X_{it}=Z_{it}+\theta Z_{i,t-1}$, which satisfies $\sum_j c_j^2=1+\theta^2$; and a causal AR(1) process $X_{it}-\phi X_{i,t-1}=Z_{it}$, $|\phi|<1$, where $\sum_j c_j^2=(1-\phi^2)^{-1}$. Yet another example of a linear process fitting in our framework is a fractionally integrated ARMA($p,d,q$) processes with $d<0$ and regularly varying noise with index $\alpha\in[1,4)$, see, e.g., \cite{Brockwell1991} for further details. In this case $|c_j|\leq C j^{d-1}$ is summable and therefore condition \eqref{summability c_j} is satisfied for $\alpha\geq 1$. 

Regarding the normalization in \eqref{pp eigenvalues}, the sequence $a_n$ is chosen such that the individual entries of the matrix $Z\coloneqq(Z_{it})_{i,t}$ satisfy \eqref{Z reg var}. Replacing the iid sequence in the rows of $Z$ with a linear process to obtain the matrix $X$ changes the tail behavior of its entries. Indeed, the result stated in Davis and Resnick \cite[eq. (2.7)]{Davis1985} shows,  under the assumption \eqref{balancing} and $EZ_{11}=0$ if $\alpha>1$, that
\begin{align*}
nP\left( \Big|\sum_j c_j Z_{1,t-j}\Big|>a_{np}^2 x \right) \ston{} x^{-\alpha} \sum_j |c_j|^\alpha.
\end{align*} 
In view of \eqref{Z reg var} this suggests the normalization $\tilde{X}_{it}={\left(\sum_j |c_j|^\alpha\right)^{-1/\alpha}}X_{it}$.
Denote by $\tilde{\lambda}_{1},\ldots,\tilde{\lambda}_p$ the eigenvalues of $\tilde{X}\tilde{X}^\T$, where $\tilde{X}=(\tilde{X}_{it})_{i,t}$, and let $\mu_{\tilde X,\alpha}=E\tilde X_{11}^2={\left(\sum_j |c_j|^\alpha\right)^{-2/\alpha}}\mu_{X,\alpha}$. Since this is just a multiplication by a constant, we immediately obtain, by \cref{th1} (i), that
\[ \sum_{i=1}^p \eps_{a_{np}^{-2}({\lambda}_{i}-n\mu_{\tilde X,\alpha})}\ston{D} \tilde{N} , \]
where $\tilde{N}$ is a Poisson process with intensity measure $\tilde{\nu}$ given by
\begin{align}
 \tilde{\nu}((x,\infty])=x^{-\alpha/2}\frac{\left|\sum_{j} c_j^2 \right|^{\alpha/2}}{\sum_j|c_j|^\alpha} .
\end{align}
Thus ${\left|\sum_{j} c_j^2 \right|^{\alpha/2}}({\sum_j|c_j|^\alpha})^{-1}$ quantifies the effect of the dependence on the point process of the eigenvalues when the tail behavior of each marginal $X_{it}$ is equivalent to the iid case.

Assume for a moment that the dimension $p$ is fixed for any $n$, and that $0<\alpha<2$. Then it follows easily from \cite[Theorem 4.1]{Davis1985} and arguments of our paper that $a_n^{-2}\lambda_{(1)}\to \sum_j c_j^2 \max_{1\leq i\leq p} S_i$ in distribution as $n\to\infty$, where $(S_i)$ are independent positive stable with index $\alpha/2$. If $p$ is large, one would intuitively expect that $\max_{1\leq i\leq p} S_i\approx p^{2/\alpha} \Gamma_1^{-2/\alpha}$, where $\Gamma_1$ is exponentially distributed with mean 1. \cref{corollary after th1} not only makes this intuition precise but also gives the correct normalization $a_{np}^{-2}$.
The distribution of the maximum of $p$ independent stables is not known analytically, hence `large n, large p' in fact gives a simpler solution than the traditional `fixed p, large n' setting.

\subsection{Extension to random coefficient models}\label{extension} 

So far we have assumed that our observed process has independent components, each of which are modelled by the same linear process. From now on we will allow for a different set of coefficients in each row. To this end, let $(\theta_i)_{i\in\N}$ be a sequence of random variables \emph{independent of $(Z_{it})$} with values in some space $\Theta$. Assume that there is a family of measurable functions $(c_j:\Theta\to\R)_{j\in\N}$ such that
\begin{align}
\sup_{\theta\in\Theta}|c_j(\theta)|\leq \wt{c}_j, \quad \tn{for some deterministic $\wt{c}_j$ satisfying condition \eqref{summability c_j}.}
\label{c_j bounded} 
\end{align}
Our observed processes have the form
\begin{align}
X_{it} = \sum_{j=-\infty}^\infty c_j(\theta_i) Z_{i,t-j}
\label{def Xit gen}
\end{align}
where $(Z_{it})$ is given as in \eqref{Z reg var} with $\alpha\in(0,4)$. Thus, conditionally on the latent process $(\theta_i)$, the rows of $X$ are independent linear processes with different coefficients. Unconditionally, the rows of $X$ are dependent if the sequence $(\theta_i)$ is dependent. \cref{th2} below covers three classes among which $(\theta_i)$ may be chosen: stationary ergodic; stationary but not necessarily ergodic; and ergodic in the Markov chain sense but not necessarily stationary. 
In the following we say that a sequence of point processes $\mathscr{M}_n$ converges, conditionally on a sigma-algebra $\mathcal{H}$, in distribution to a point process $\mathscr{M}$, if the conditional Laplace functionals converge almost surely, i.e., if there exists a measurable set $B$ with $P(B)=1$ such that for all $\omega\in B$ and all nonnegative continuous functions $f$ with compact support,
\begin{align}
 E\left(e^{-\mathcal{M}_n (f)}\big|\mathcal{H}\right)(\omega)\ston{} E\left(e^{-\mathcal{M} (f)}\big|\mathcal{H}\right)(\omega) \quad\textnormal{as }n\to\infty.
\label{cond pp convergence}\end{align}

\begin{theorem}\label{th2}
Define $X=(X_{it})$ with $X_{it}$ as given in \eqref{def Xit gen}. Suppose that \eqref{c_j bounded} is satisfied, and $p,n\to\infty$ such that \eqref{beta con}  holds under the same conditions as in \cref{th1} (i). Further assume, in case $\alpha\in(5/3,4)$, that $Z_{11}$ has mean zero and satisfies the tail balancing condition \eqref{balancing}.
\begin{enumerate}
 \item If $(\theta_i)$ is a stationary ergodic sequence, then, both conditionally on $(\theta_i)$ as well as unconditionally, we have that
 \begin{align}
\sum_{i=1}^p  \eps_{a_{np}^{-2}(\lambda_i-n\mu_{X,\alpha})} \ston{D} \sum_{i=1}^\infty \eps_{\Gamma_i^{-2/\alpha} \norm{\sum_{j} c_j^2(\theta_1)}_{\frac{\alpha}{2}} },
\label{mainClaimTh2}\end{align}
with the constant $\norm{\sum_{j} c_j^2(\theta_1)}_{{\frac{\alpha}{2}}}=\left(E\left|\sum_{j} c_j^2(\theta_1)\right|^{\alpha/2}\right)^{2/\alpha}$, and $(\Gamma_i)$ as in \eqref{N via Gamma}.
\item If $(\theta_i)$ is stationary but not necessarily ergodic, then we have, conditionally on $(\theta_i)$, that
\[ \sum_{i=1}^p  \eps_{a_{np}^{-2}(\lambda_i-n\mu_{X,\alpha})} \ston{D} \sum_{i=1}^\infty \eps_{\Gamma_i^{-2/\alpha} Y^{2/\alpha}}, \]
with $Y=E\left(|\sum_j{c_j^2(\theta_1)}|^{\alpha/2}|\mathcal{G}\right)$, where $\mathcal{G}$ is the invariant $\sigma$-field generated by $(\theta_i)$. In particular, $Y$ is independent of $(\Gamma_i)$.
\item Suppose $(\theta_i)$ is either an irreducible Markov chain on a countable state space $\Theta$ or a positive Harris chain in the sense of Meyn and Tweedie \cite{Meyn2009}. If $(\theta_i)$ has a stationary probability distribution $\pi$ then, conditionally on $(\theta_i)$ as well as unconditionally, \eqref{mainClaimTh2} holds with
\[ \norm{\sum_{j} c_j^2(\theta_1)}_{{\frac{\alpha}{2}}}=\left(\int_\Theta{\Big|\sum_j c_j^2(\theta)\Big|^{\alpha/2} \pi(d\theta)}\right)^{2/\alpha}.\]
\end{enumerate}
\end{theorem}

One can view the assumptions (i) and (ii) of \cref{th2} in a Bayesian framework in which the parameters of the observed process are drawn from an unknown prior distribution. As an example, let $(\theta_i)$ be a stationary ergodic AR(1) process  $\theta_i=\phi\theta_{i-1}+\xi_i$, where $|\phi|\neq 1$ and $(\xi_i)$ is a sequence of bounded iid random variables, and set $X_{it}=Z_{it}+\theta_i Z_{i,t-1}$. Then, by \cref{th2} (i), we would expect, for $n$ and $p$ large enough, that 
\[ a_{np}^{-2}(\lambda_{(1)}-n\mu_{X,\alpha}) \approx \Gamma_1^{-\alpha/2} \left(E\left| 1+ \theta_1 \right|^{\alpha/2}\right)^{2/\alpha}. \]
Models of this kind are referred to as \emph{random coefficient models} and often used in times series analysis, see, e.g., \cite{Longford1993} for an overview.
In the setting of \cref{th2} (iii) one might think of a \emph{Hidden Markov Model} where the latent Markov process $(\theta_i)$ evolves along the rows of $X$,
each state $\theta_i$ defining another univariate linear model.

\section{Proofs and auxiliary results}

The first step is to show that the matrix $XX^\T$ is well approximated by its diagonal, see \cref{opnorm}. In the second step we then derive the extremes of the diagonal of $XX^\T$ in \cref{extremes diagonal}. Both steps together yield the proofs of \cref{th1} and \cref{th1b} in \cref{proof th1}. The proof of \cref{th2} follows then by an extension of the previous methods in \cref{proof extension}. In the following we make frequent use of a large deviation result which is presented in the upcoming section.

\subsection{A large deviation result and its consequences}\label{sec large dev}

The next theorem gives the joint large deviations of the sum and the maximum of iid nonnegative random variables with infinite variance. It suffices to deal with the case where $0<\alpha<2$ since later on we mostly consider squared random variables that have tail index $\alpha/2$ with $0<\alpha/2<2$.

\begin{proposition}\label{large deviation}
Let $(x_n)_{n\in\N}$ and $(y_n)_{n\in\N}$ be sequences of nonnegative numbers with $x_n\to\infty$ such that $x_n/y_n\to\gamma\in(0,\infty]$. Suppose $(Y_t)_{t\in\N}$ is an iid sequence of nonnegative random variables with tail index $\alpha\in(0,2)$ and normalizing sequence $b_n$.
If $1\leq\alpha<2$, we assume that $b_n x_n/n^{1+\delta}\to\infty$ for some $\delta>0$. Then
\begin{align}
\lim_{n\to\infty}\frac{P\left(\sum_{t=1}^n Y_t > b_n x_n,\max_{1\leq t\leq n} Y_t > b_n y_n\right)}{nP(Y_1>b_n\max\{x_n,y_n\})}=1.
\label{large dev result}
\end{align}
\end{proposition}


\begin{proof}
Let us first assume that $0<\alpha<1$.
Using standard arguments from the theory of regularly varying functions, see e.g. \cite{Resnick2008},
it can be easily seen that for any positive sequence $z_n\to\infty$ we have
\begin{align}
\lim_{n\to\infty}\frac{P(\max_{1\leq t\leq n} Y_t > b_n z_n)}{nP(Y_1>b_n z_n)}=1.
\label{large dev max}
\end{align}
Obviously the limit in \eqref{large dev result} is greater or equal than one because $\sum_{t=1}^n Y_t\geq\max_{1\leq t\leq n} Y_t$.
Thus it is only left to prove that it is also smaller. Denote by $Y_{(1)}\geq\ldots\geq Y_{(n)}$ the upper order statistics of $Y_1,\ldots,Y_n$. By decomposing $\sum_t Y_t$ into the sum of $\max_t Y_t$ and lower order terms we see that, for any $\theta\in(0,1)$,
\begin{align*}
\frac{P\left(\sum_{t=1}^n Y_t > b_n x_n,\max_{1\leq t\leq n} Y_t > b_n y_n\right)}{nP(Y_1>b_n\max\{x_n,y_n\})} \leq& \frac{P(\max_{1\leq t\leq n} Y_t > b_n \max\{\theta x_n,y_n\})}{nP(Y_1>b_n \max\{x_n,y_n\})} + \frac{P\left(\sum_{t=2}^n Y_{(t)} > b_n x_n(1-\theta)\right)}{nP(Y_1>b_n\max\{x_n,y_n\})}. 
\end{align*}
By an application of \cite[Proposition 0.8 (iii)]{Resnick2008} one can show similarly as in the 
proof of \eqref{large dev max} that 
\[ \lim_{\theta\to 1}\lim_{n\to\infty} \frac{P(\max_{1\leq t\leq n} Y_t > b_n \max\{\theta x_n,y_n\})}{nP(Y_1>b_n \max\{x_n,y_n\})} = 1. \]
Hence, it is only left to show that the second summand vanishes as $n\to\infty$. To this end we partition the underlying probability space into $\{Y_{(2)}\leq\epsilon b_n x_n\}\cup\{Y_{(2)}>\epsilon b_n x_n\}$, $\epsilon>0$, to obtain
\begin{align*}
\frac{P\left(\sum_{t=2}^n Y_{(t)} > b_n x_n(1-\theta)\right)}{nP(Y_1>b_n\max\{x_n,y_n\})} \leq&
\frac{P\left(\sum_{t=2}^n Y_{(t)} \1_{\{Y_{(2)}\leq\epsilon b_n x_n\}} > b_n x_n(1-\theta)\right)}{nP(Y_1>b_n\max\{x_n,y_n\})} \nonumber\\
&+ \frac{P\left( Y_{(2)}>\epsilon b_n x_n \right)}{nP(Y_1>b_n\max\{x_n,y_n\})}
=\Sigma_1 + \Sigma_2.
\end{align*}
Denote by $M_n=\max_{1\leq t\leq n} Y_t$ and $z_n=\max\{x_n,y_n\}$.
Then easy combinatorics and \eqref{large dev max} yield
\begin{align*}
\Sigma_2 =& \frac{1-P\left( Y_{(2)}\leq\epsilon b_n x_n \right)}{nP(Y_1>b_n z_n)} \\
 =& \frac{1-P\left( M_n \leq\epsilon b_n x_n \right)}{nP(Y_1>b_n z_n)} - \frac{nP\left( M_{n-1} \leq\epsilon b_n x_n \right) P(Y_1>\epsilon b_n x_n )}{nP(Y_1>b_n z_n)} \\
=& \frac{P\left( M_n > \epsilon b_n x_n \right)}{nP( Y_1 > \epsilon b_n x_n)}\frac{P( Y_1 > \epsilon b_n x_n)}{P(Y_1>b_n z_n)} - \frac{P(Y_1>\epsilon b_n x_n )}{P(Y_1>b_n z_n)}P\left(M_{n-1} \leq\epsilon b_n x_n\right) \\
\sim& \frac{P(Y_1>\epsilon b_n x_n )}{P(Y_1>b_n z_n)} \left( 1 - P\left(M_{n-1} \leq\epsilon b_n x_n\right) \right) \ston{} 0.
\end{align*}
The convergence to zero follows from $P\left(M_{n-1} \leq\epsilon b_n x_n\right)\to 1$ and, by \cite[Proposition 0.8 (iii)]{Resnick2008}, $\frac{P(Y_1>\epsilon b_n x_n )}{P(Y_1>b_n z_n)}\to\epsilon^{-\alpha}\max\{1,\gamma^{-\alpha}\}$. Thus it is only left to show that $\Sigma_1$ goes to zero. By Markov's inequality and Karamata's Theorem \cite[Theorem 0.6]{Resnick2008} we have that
\begin{align*}
\Sigma_1\leq& \frac{P\left(\sum_{t=1}^n Y_{t} \1_{\{Y_{t}\leq\epsilon b_n x_n\}} > b_n x_n(1-\theta)\right)}{nP(Y_1>b_n\max\{x_n,y_n\})} \\ \leq& \frac{1}{b_n x_n (1-\theta)}\frac{E(Y_1\1_{\{Y_{1}\leq\epsilon b_n x_n\}})}{P(Y_1>b_n z_n)} \sim \frac{1}{(1-\theta)} \frac{\alpha}{1-\alpha} \frac{\epsilon P(Y_{1}>\epsilon b_n x_n)}{P(Y_1>b_n z_n)} \ston{} \frac{1}{(1-\theta)} \frac{\alpha}{1-\alpha} \epsilon^{1-\alpha} \max\{1,\gamma^{-\alpha}\}, 
\end{align*}
which converges to zero as $\epsilon$ goes to zero, since $\alpha<1$. Thus for $0<\alpha<1$ the proof is complete. If $1\leq\alpha<2$, only  $\Sigma_1$ has to be treated differently. 
The truncated mean $\mu_n=E(Y_1\1_{\{Y_1\leq \epsilon b_n x_n\}})$ either converges to a constant or is a slowly varying function. In either case, we have that $b_n x_n/(n \mu_n)=b_n x_n n^{-1-\delta}\, n^\delta/\mu_n\to\infty$ by assumption. Thus, a mean-correction argument and Karamata's Theorem imply
\begin{align*}
\limsup_{n\to\infty} \Sigma_1 \leq& 
 \limsup_{n\to\infty}\frac{P\left(\sum_{t=1}^n Y_{t} \1_{\{Y_{t}\leq\epsilon b_n x_n\}} - n\mu_n > b_n x_n(1-\theta) - n\mu_n \right)}{nP(Y_1>b_n z_n)} \\
\leq& \frac{1}{(1-\theta)^2} \limsup_{n\to\infty}\frac{1}{b_n^2 x_n^2}\frac{\var(Y_1\1_{\{Y_{1}\leq\epsilon b_n x_n\}})}{P(Y_1>b_n z_n)} 
\leq \frac{1}{(1-\theta)^2} \frac{\alpha/2}{1-\alpha/2}  \epsilon^{2-\alpha}  \max\{1,\gamma^{-\alpha}\} \sto{}{\epsilon\to 0} 0,
\end{align*}
since $\alpha<2$. This completes the proof.
\end{proof}

We finish this section with a few consequences of \cref{large deviation}. Note that \eqref{Z reg var} implies
\begin{align}\label{a_np}
pn P(Z_{11}^2>a_{np}^2x) \ston{} x^{-\alpha/2} \quad\tn{ for each } x>0.
\end{align}
Choosing $Y_t=Z_{1t}^2$, $b_n=a_n^2$, $x_n=xa_{np}^2/a_n^2$ and $y_n=ya_{np}^2/a_n^2$, we have from \cref{large deviation} and \eqref{a_np}, for $\alpha\in(0,2)$, that
\begin{align*}
pP\left(\sum_{t=1}^n Z_{1t}^2 > a_{np}^2 x ,\max_{1\leq t\leq n} Z_{1t}^2 > a_{np}^2 y\right) \ston{} \max\{x,y\}^{-\alpha/2}  \quad\tn{ for each } x,y>0.
\end{align*}
Therefore, by \cite[Proposition 3.21]{Resnick2008}, we obtain the point process convergence
\begin{align}
\sum_{i=1}^p \eps_{a_{np}^{-2}(\sum_{t=1}^n Z_{it}^2,\max_{1\leq t\leq n} Z_{it}^2)} \ston{D} \sum_{i=1}^\infty \eps_{{\Gamma}_i^{-2/\alpha}(1,1)},
\label{pp conv max and sum sq}\end{align}
with $(\Gamma_i)$ as in \eqref{N via Gamma}. For another application of \cref{large deviation}, set $Y_t=|Z_{1t}|$, $b_n=a_n$, $x_n=xa_{np}/a_n$ and $y_n=ya_{np}/a_n$. 
Under the additional assumption
\[ \liminf_{n\to\infty} \frac{p}{n} \in (0,\infty] \]
we have $b_n x_n/n^{1+\gamma}\to\infty$ for some $\gamma<(2-\alpha)/\alpha$, thus, for $\alpha\in(0,2)$,
\begin{align*}
pP\left(\sum_{t=1}^n |Z_{1t}| > a_{np} x ,\max_{1\leq t\leq n} |Z_{1t}| > a_{np} y\right) \ston{} \max\{x,y\}^{-\alpha}  \quad\tn{ for each } x,y>0.
\end{align*}
Therefore we obtain as before
\begin{align}
\sum_{i=1}^p \eps_{a_{np}^{-1}(\sum_{t=1}^n |Z_{it}|,\max_{1\leq t\leq n} |Z_{it}|)} \ston{D} \sum_{i=1}^\infty \eps_{{\Gamma}_i^{-1/\alpha}(1,1)}.
\label{pp conv max and sum abs}\end{align}
The result of the following proposition is also a consequence of \cref{large deviation}.

\begin{proposition}
Let $(Z_{it})$ be as in \eqref{Z reg var} with $0<\alpha<2$. Suppose that \eqref{beta con} is satisfied for some $0<\beta<\infty$. Then
\[ a_{np}^{-2} \max_{1\leq i<j\leq p} \sum_{t=1}^n |Z_{it} Z_{jt}| \ston{P} 0. \]
\label{lemma non-diag elemements}
\end{proposition}

\begin{proof}
By \cite{Embrechts1980}, the iid random variables $Y_t=|Z_{1t}Z_{2t}|$ are regularly varying with tail index $\alpha$ with some normalizing sequence $b_n$. Thus, there exists a slowly varying $L_1$ such that $P(Y_1>x)=x^{-\alpha}L_1(x)$. Using \eqref{a_n} this implies
\[ p^2nP(Y_1>a_{np}^2 \epsilon ) = n^{-1} \epsilon^{-\alpha} \, L(np)^{-2\alpha} \, L_1\left((np)^{2/\alpha}L(np)^2\epsilon\right). \]
By Potter's bound, see, e.g., \cite[Proposition 0.8 (ii)]{Resnick2008}, for any slowly varying function $\tilde{L}$ and any $\delta>0$ there exist $c_1,c_2>0$ such that $c_1 n^{-\delta} < \tilde{L}(n) < c_2 n^{\delta}$ for $n$ large enough. An application of this bound together with assumption \eqref{beta con} shows that
\begin{align}
p^2 n P(Y_1> a_{np}^2 \epsilon) \ston{} 0.
\label{a_np2strong}
\end{align}
Hence, using \cref{large deviation} with $x_n=a_{np}^2/b_n \epsilon$ and $y_n=0$ yields
\[ P\left(\max_{1\leq i<j\leq p} \sum_{t=1}^n |Z_{it} Z_{jt}| > a_{np}^2 \epsilon\right) \leq p^2 P\left( \sum_{t=1}^n Y_t > a_{np}^2 \epsilon \right) \ston{} 0, \]
since $b_n x_n/n^{1+\gamma} = a_{np}^2/n^{1+\gamma}\to \infty$ for $\alpha<2$ and some $\gamma<(2-\alpha)/\alpha$.
\end{proof}

\subsection{Convergence in Operator Norm}\label{opnorm}

Denote by $D=\diag(XX^\T)$ the diagonal of the matrix $XX^\T$, i.e., $D_{ii}=(XX^\T)_{ii}$ and $D_{ij}=0$ for $i\neq j$. In this section we show that $a_{np}^{-2}(XX^\T-D)$ converges in probability to 0 in operator norm. This implies that the off-diagonal elements of $a_{np}^{-2}XX^\T$ do not contribute to the limiting eigenvalue spectrum. Recall that, for a real $p\times n$ matrix $A$, the operator $2$-norm $\norm{A}_2$ is the square root of the largest eigenvalue of $AA^\T$, and the infinity-norm is given by $\norm{A}_\infty=\max_{1\leq i\leq p}\sum_{t=1}^n |A_{it}|$.\\

In the upcoming \cref{th conv op norm} we only deal with the case where $0<\alpha<2$. Note that \cref{th conv op norm} holds under a much more general setting than assumed in \cref{th1} (i) by allowing for an arbitrary dependence structure within the rows of $X$.

\begin{proposition}\label{th conv op norm}
Let $X=(X_{it})_{i,t}$ be a $p\times n$ random matrix whose entries are identically distributed with tail index $\alpha\in(0,2)$ and normalizing sequence $(a_n)$. Assume that the rows of $X$ are independent. 
Suppose that \eqref{beta con} holds for some $\beta>0$. If $1<\alpha<2$, assume additionally that $\beta<\frac{2-\alpha}{\alpha-1}$. Then we have 
\begin{align}
a_{np}^{-2} \norm{XX^\T-D}_2 \ston{P} 0.
\label{eq conv opnorm}
\end{align}
\end{proposition}

\begin{proof}
Since $\norm{XX^\T-D}_2\leq\norm{XX^\T-D}_\infty$, it is enough to show that for every $\epsilon\in(0,1)$,
\begin{align*}
P\left( \max_{i=1,\ldots,p} \sum_{\substack{j=1 \\ j\neq i}}^p \left| \sum_{t=1}^n X_{it} X_{jt} \right| > a_{np}^2 \epsilon \right) &\leq p P\left( \sum_{j=2}^p \sum_{t=1}^n |X_{1t} X_{jt}| > a_{np}^2 \epsilon \right)\ston{}0.
\end{align*}
By partitioning the underlying probability space into $\{\max_{j,t}|X_{1t}X_{jt}|\leq a_{np}^2\}$ and its complement, we obtain that
\begin{align*}
p P\left( \sum_{j=2}^p \sum_{t=1}^n |X_{1t} X_{jt}| > a_{np}^2 \epsilon \right) \leq&
p P\left( \sum_{j=2}^p \sum_{t=1}^n |X_{1t} X_{jt}| \1_{\{|X_{1t} X_{jt}|\leq {a}_{np}^2\}}  > {a_{np}^2}\epsilon \right) \\ &+ p P\left( \max_{2\leq j\leq p} \max_{1\leq t\leq n} |X_{1t} X_{jt}|  > {a}_{np}^2\epsilon \right) = \mrm{I} + \mrm{II}.
\end{align*}
The same argument used for \eqref{a_np2strong} shows that $\mrm{II} \leq p^2 n P( |X_{11} X_{21}|> a_{np}^2 ) \ston{} 0$ by independence of the rows of $X$. To deal with term $\mrm{I}$ we   first assume that $\alpha>1$ and choose some $\gamma\in(\alpha,2)$. Hölder's inequality shows that
\[ \left(\sum_{j=2}^p \sum_{t=1}^n |X_{1t} X_{jt}|\right)^\gamma \leq \left(\sum_{j=2}^p \sum_{t=1}^n |X_{1t} X_{jt}|^\gamma\right) (np)^{\gamma-1}, \]
and therefore
\[ \mathrm{I}\leq p P\left( \sum_{j=2}^p \sum_{t=1}^n |X_{1t} X_{jt}|^\gamma \1_{\{|X_{1t} X_{jt}|\leq {a}_{np}^2\}}  > \frac{a_{np}^{2\gamma}}{(np)^{\gamma-1}} \epsilon \right). \]
Note that $|X_{1t} X_{jt}|^\gamma $ has regularly varying tails with index $\alpha/\gamma<1$. Hence we can apply Markov's Inequality and Karamata's Theorem to infer that
\begin{align}
\mrm{I} &\leq 
 c_1 \frac{p^2n (np)^{\gamma-1}}{a_{np}^{2\gamma}} E\left( |X_{11} X_{21}|^\gamma \1_{\{|X_{11} X_{21}|\leq {a}_{np}^2\}} \right) \sim c_2 p^2 n (np)^{\gamma-1} P(|X_{11} X_{21}| > {a}_{np}^2). \label{eqOpNormProof}
\end{align}
Therefore, the proof of \cref{lemma non-diag elemements} shows that the term in \eqref{eqOpNormProof} goes to zero if $(np)^{\gamma-1}/n$ does. In view of assumption \eqref{beta con} this is true for $\beta<{(2-\gamma)}/({\gamma-1})$. 
Since we can choose $\gamma$ arbitrary close to $\alpha$ it suffices that $\beta<{(2-\alpha)}/({\alpha-1})$. 
If $\alpha<1$ we do not need Hölder's inequality since the above argument can be applied with $\gamma=1$, thus it suffices that \eqref{beta con} holds for some $\beta<\infty$. For the remaining case $\alpha=1$, observe that, for any given $\beta<\infty$, we choose $\gamma$ arbitrarily close to $1$ so that $(np)^{\gamma-1}/n\to 0$.
\end{proof}

The next proposition improves the previous \cref{th conv op norm} for $5/3<\alpha<2$ at the expense of the additional assumption that the rows of $X$ are realizations of a linear process. Furthermore, \cref{th conv op norm 2} also covers the case where $2\leq\alpha<4$.

\begin{proposition}\label{th conv op norm 2}
The assumptions of \cref{th1} (i) imply \eqref{eq conv opnorm}.
\end{proposition}

\begin{proof}
In this proof, $c$ denotes a positive constant that may vary from expression to expression.
Define
\begin{align*}
Z_{it}^{L} &= Z_{it} \1_{\{ |Z_{it}|\leq a_{np} \}}, \quad X_{it}^{L} = \sum_k c_k Z_{i,t-k}^{L}, \\
Z_{it}^{U} &= Z_{it} \1_{\{ |Z_{it}|> a_{np} \}}, \quad X_{it}^{U} = \sum_k c_k Z_{i,t-k}^{U}.
\end{align*}
Using $\norm{XX^\T-D}_2\leq\norm{XX^\T-D}_\infty$ as before we have
\begin{align*}
P\left(\norm{XX^\T-D}_2 > a_{np}^{2}\epsilon\right) \leq& p P\left(\sum_{j=2}^p \left| \sum_{t=1}^n X_{1t} X_{jt} \right| > a_{np}^2\epsilon \right) \\
\leq& p P\left(\sum_{j=2}^p \left| \sum_{t=1}^n X_{1t}^{L} X_{jt}^{L} \right| > \frac{a_{np}^2}{4}\epsilon \right) + p P\left(\sum_{j=2}^p \left| \sum_{t=1}^n X_{1t}^{L} X_{jt}^{U} \right| > \frac{a_{np}^2}{4}\epsilon \right) \\
&+ p P\left(\sum_{j=2}^p \left| \sum_{t=1}^n X_{1t}^{U} X_{jt}^{L} \right| > \frac{a_{np}^2}{4}\epsilon \right) + p P\left(\sum_{j=2}^p \left| \sum_{t=1}^n X_{1t}^{U} X_{jt}^{U} \right| > \frac{a_{np}^2}{4}\epsilon \right) \\
=& \mathrm{I} + \mathrm{II} + \mathrm{III} + \mathrm{IV}.
\end{align*}
We will show that each of theses terms converges to zero. To this end, note that $E|Z_{11}^{L}|$ converges to a constant, and, by Karamata's Theorem, 
\[ E|Z_{11}^{U}| \sim c a_{np} P(|Z_{11}|>a_{np}) \sim c a_{np} (np)^{-1}, \quad n\to\infty. \]
Therefore, by Markov's inequality, we have
\begin{align*}
\mathrm{II}\leq \frac{4p}{a_{np}^2\epsilon} \sum_{j=2}^p \sum_{t=1}^n \sum_{k,l} |c_k c_l| \, E|Z_{1,t-k}^{L}| \, E|Z_{j,t-l}^{U}| \sim c \left(\sum_k |c_k|\right)^2 \frac{p^2 n}{a_{np}^2} a_{np} (np)^{-1} = c \frac{p}{a_{np}}, 
\end{align*}
and, by \eqref{a_n}, we obtain that this is equal to $c \, L(np)^{-1} p^{1-1/\alpha} n^{-1/\alpha} \to 0$ as $n\to\infty$. By symmetry, $\mathrm{III}$ can be handled the same way. It is easy to see that term $\mathrm{IV}$ is of even lower order, namely
\[ c \frac{p^2 n}{a_{np}^2} (a_{np} (np)^{-1})^{2} = c n^{-1} \to 0. \]
Thus it is only left to show that $\mrm{I}$ converges to zero. To this end, we use Karamata's Theorem to obtain
\begin{align*}
E\left[(Z_{11}^{L})^2\right]=E\left[Z_{11}^2\1_{\{ |Z_{11}|\leq a_{np} \}}\right] \sim c a_{np}^2 P(|Z_{11}|>a_{np}) \sim c a_{np}^2 (np)^{-1}.
\end{align*}
Since $Z_{11}$ satisfies the tail balancing condition \eqref{balancing}, and $EZ_{11}=0$, we can apply Karamata's Theorem to the positive and the negative tail of $Z_{11}^L$, thus, for $q\notin\{0,\frac12,1\}$,
\begin{align*}
\xi_n \coloneqq& E[Z_{11}^{L}]=E[Z_{11}\1_{\{ |Z_{11}|\leq a_{np} \}}] = - E[Z_{11}\1_{\{ |Z_{11}|> a_{np} \}}] = - E[Z_{11}\1_{\{ Z_{11}> a_{np} \}}] + E[-Z_{11}\1_{\{ - Z_{11}> a_{np} \}}] \\
\sim& - q \frac{\alpha}{\alpha-1} a_{np}  P(|Z_{11}|> a_{np}) + (1-q) \frac{\alpha}{\alpha-1} a_{np} P(|Z_{11}|> a_{np}) \sim (1-2q) \frac{\alpha}{\alpha-1} a_{np} (np)^{-1}.
\end{align*}
Clearly, for any $0\leq q\leq 1$, one therefore has
\[ \frac{np\xi_n}{a_{np}} \to (1-2q)\frac{\alpha}{\alpha-1}. \]
As a consequence we obtain for $\mu_n=E(X_{11}^{L} X_{21}^{L}) = (EX_{11}^{L})^2=\xi_n^2$ that
\[ \frac{\mu_n pn}{a_{np}^2} = (np)^{-1} \left(\frac{np\xi_n}{a_{np}}\right)^2 \left(\sum_k c_k\right)^2 \to 0. \]
Therefore we obtain for summand $\mrm{I}$ that
\begin{align*}
\mathrm{I} = p P\left(\sum_{j=2}^p \left| \sum_{t=1}^n X_{1t}^{L} X_{jt}^{L} \right| > \frac{a_{np}^2}{4}\epsilon \right) \leq& p^2 P\left(\left| \sum_{t=1}^n X_{1t}^{L} X_{2t}^{L} \right| > \frac{a_{np}^2}{4p}\epsilon \right) \sim p^2 P\left(\left| \sum_{t=1}^n{X_{1t}^{L} X_{2t}^{L}} - n\mu_n \right| > \frac{a_{np}^2}{4p}\epsilon \right).
\end{align*}
Therefore, it is only left to show that
\begin{align}\label{mainterm}
p^2 P\left(\left| \sum_{t=1}^n{X_{1t}^{L} X_{2t}^{L}} - n\mu_n \right| > \frac{a_{np}^2}{4p}\epsilon \right)\to 0,
\end{align}
with
\[ \mu_n = (EX_{11}^L)^2 = \left(\sum_k c_k\right)^2 (EZ_{11}^L)^2 = O\left(\frac{a_{np}^2}{(np)^2}\right). \]
Now we have to treat the cases $\alpha<2$, $2\leq\alpha<3$, and $3\leq\alpha<4$ separately.\\

Let $\alpha<2$. 
By \cref{th conv op norm} it suffices to show that, for $\alpha\in(5/3,2)$, the assumption 
\begin{align}
\lim_{n\to\infty} \frac{p}{\sqrt{n}}=0
\end{align}
implies convergence in operator norm in the sense of \eqref{eq conv opnorm}.
Since we correct by the mean, Markov's inequality yields
\begin{align}
p^2  P\left(\left| \sum_{t=1}^n{X_{1t}^{L} X_{2t}^{L}} - n\mu_n \right| > \frac{a_{np}^2}{4p}\epsilon \right) \leq& \frac{16p^4}{a_{np}^4\epsilon^2} \var\left(\sum_{t=1}^n X_{1t}^{L} X_{2t}^{L}\right) \nonumber\\
 =& \frac{16p^4}{a_{np}^4\epsilon^2} \sum_{t,t'=1}^n \sum_{k,k',l,l'} c_k c_{k'} c_l c_{l'} \cov\left(Z_{1,t-k}^{L}Z_{2,t-l}^{L},Z_{1,t'-k'}^{L}Z_{2,t'-l'}^{L}\right). \label{sum cov}
\end{align}
Due to the independence of the $Z$'s, the covariance in the last expression is  non-zero iff $t-k= t'-k'$ or $t-l= t'-l'$. This gives us three distinct cases we deal with separately. First, assume that both $t-k= t'-k'$ and $t-l= t'-l'$. Then the covariance in \eqref{sum cov} is equal to $\var(Z_{11}^{L}Z_{2,1}^{L})$ and so bounded by
\begin{align*}
 E[(Z_{11}^{L})^2(Z_{21}^{L})^2] = (E(Z_{11}^{L})^2)^2 \sim ( c a_{np}^2 (np)^{-1} )^2 \sim c a_{np}^4 (np)^{-2}.
\end{align*}
Second, let $t-k= t'-k'$ but $t-l\neq t'-l'$. Then the covariance becomes
\begin{align*}
\cov(Z_{1,t-k}^{L}Z_{2,t-l}^{L},Z_{1,t'-k'}^{L}Z_{2,t'-l'}^{L}) =& E((Z_{1,t-k}^{L})^2 Z_{2,t-l}^{L} Z_{2,t'-l'}^{L}) - \xi_n^4\\
=& E((Z_{1,t-k}^{L})^2) \xi_n^2 - \xi_n^4 \\
\sim& c a_{np}^2 (np)^{-1} ( \pm c \, a_{np} (np)^{-1} )^2 - ( \pm c \, a_{np} (np)^{-1} )^4 \\
\sim& c a_{np}^4 (np)^{-3} - c a_{np}^4 (np)^{-4} \sim c a_{np}^4 (np)^{-3},
\end{align*}
which is of lower order than in the case considered before. By symmetry, the third case, where $t-l= t'-l'$ but $t-k\neq t'-k'$, can be dealt with in exactly the same way. In all cases $t'$ can be assumed to be fixed, thus we can bound \eqref{sum cov} by
\[ c \frac{p^4}{a_{np}^4} \left(\sum_k |c_k|\right)^4 \sum_{t=1}^n a_{np}^4 (np)^{-2} = c \frac{p^2}{n} \to 0, \quad n\to\infty.  \]
This completes the proof in case $\alpha<2$.\\

If $\alpha>2$, the covariance in \eqref{sum cov} converges to a constant. If $\alpha=2$ with $EZ_{11}^2=\infty$, then it is a slowly varying function. In either case \eqref{sum cov} is of order
\[ O\left( \frac{p^4}{a_{np}^4} n \,s(np) \right) \leq O\left( n^{\beta(4-4/\alpha)} n^{1-4/\alpha} s(np) \right) \to 0, \]
since $\beta<\frac{4-\alpha}{4(\alpha-1)}$, where $s(\cdot)$ is some slowly varying function. For a more general result we distinguish the sub-cases $\alpha\in(2,3)$ and $\alpha\in[3,4)$ in the following.
\\

Let us now assume that $\alpha\in(2,3)$. By Markov's inequality applied to \eqref{mainterm} we have
\begin{align}
p^2 & P\left(\left| \sum_{t=1}^n{X_{1t}^{L} X_{2t}^{L}} - n\mu_n \right| > \frac{a_{np}^2}{4p}\epsilon \right) \nonumber\\ \leq& \frac{64}{\epsilon^3}\frac{p^5}{a_{np}^6}\sum_{t_1,t_2,t_3=1}^n E\left( \prod_{i=1}^3 \left( X_{1,t_i}^L X_{2,t_i}^L - \mu_n \right) \right)  \nonumber\\ =& \frac{64}{\epsilon^3}\frac{p^5}{a_{np}^6}\sum_{t_1,t_2,t_3=1}^n \sum_{k_1,k_2,k_3}\sum_{l_1,l_2,l_3} \prod_{j=1}^3 (c_{k_j}c_{l_j}) E\left( \prod_{i=1}^3 \left( Z_{1,t_i-k_i}^L Z_{2,t_i-l_i}^L - \xi_n^2 \right) \right), \label{crazyexpectation}
\end{align}
where
\begin{align}\label{ndefxi}
\xi_n^2 = \frac{\mu_n}{\left(\sum_k c_k\right)^2} = (EZ_{11}^L)^2 = O\left(\frac{a_{np}^2}{(np)^2}\right).
\end{align}
To determine the order of the expectation in \eqref{crazyexpectation} we have to distinguish various cases. 
In the following we say that two index pairs $(a,b)$ and $(c,d)$ overlap if $a=c$ or $b=d$. 
If there exists a $j=1,2,3$ such that the index pair $(t_j-k_j,t_j-l_j)$ does not overlap with both the other two, then, due to independence, we are able to factor out the corresponding term and obtain
\[ E\left( \prod_{i=1}^3 \left( Z_{1,t_i-k_i}^L Z_{2,t_i-l_i}^L - \xi_n^2 \right) \right) = E\left( \prod_{i\neq j} \left( Z_{1,t_i-k_i}^L Z_{2,t_i-l_i}^L - \xi_n^2 \right) \right) E\left( Z_{1,t_j-k_j}^L Z_{2,t_j-l_j}^L - \xi_n^2 \right) = 0, \]
since $\xi_n^2=(EZ_{11}^L)^2 =E\left( Z_{1,t_j-k_j}^L Z_{2,t_j-l_j}^L \right)$. Thus, in any non-trivial case, each index pair does overlap with (at least) one of the other two. Therefore we have at least two equalities of the form $t_i-k_i = t_{(i+1)\mrm{mod}\, 3}-k_{(i+1)\mrm{mod}\, 3} \tn{ or } t_i-l_i = t_{(i+1)\mrm{mod}\, 3}-l_{(i+1)\mrm{mod}\, 3} \tn{ for } i=1,2,3$. Hence $t_2$ and $t_3$ are immediately determined by some linear combination of $t=t_1$ and the $k_i's$ or $l_i's$. Therefore the triple sum $\sum_{t_1,t_2,t_3=1}^n$ is, if we only count terms where the covariance is non-zero, in fact a simple sum $\sum_{t=1}^n$ and so only has a contribution of order $n$.
Now we have to determine the order of the products $E\left( \prod_{i=1}^3 Z_{1,t_i-k_i}^L Z_{2,t_i-l_i}^L \right)$. If we only have a single power then, by \eqref{ndefxi}, this gives us
\[ E\left( Z_{1,t_i-k_i}^L Z_{2,t_i-l_i}^L \right) = \xi_n^2 = o(1). \]
Since $\alpha>2$, powers of order two converge to a constant,
\[ E\left( \left(Z_{1,t_i-k_i}^L Z_{2,t_i-l_i}^L\right)^2 \right) \to \var(Z_{11})^2. \]
An application of Karamata's theorem yields that
\[ E\left( \left(Z_{1,t_i-k_i}^L Z_{2,t_i-l_i}^L\right)^3 \right) \sim a_{np}^6 (np)^{-2}. \]
Using the above facts, it is easy to see that
\begin{align} E\left( \prod_{i=1}^3 Z_{1,t_i-k_i}^L Z_{2,t_i-l_i}^L \right) = O\left( a_{np}^6 (np)^{-2} \right). \label{prodexp}\end{align}
Thus we have, using \eqref{ndefxi} and \eqref{prodexp}, for the expectation in \eqref{crazyexpectation} that
\begin{align*}
E\left( \prod_{i=1}^3 \left( Z_{1,t_i-k_i}^L Z_{2,t_i-l_i}^L - \xi_n^2 \right) \right) =& \sum_{k=0}^3 (-1)^{k} \sum_{J\subseteq\{1,2,3\},|J|=k}  E\left( \prod_{i\in\{1,2,3\}\backslash J}  Z_{1,t_i-k_i}^L Z_{2,t_i-l_i}^L \right) \xi_n^{2|J|} \\
=& O\left( a_{np}^6 (np)^{-2} - \frac{a_{np}^2}{(np)^2} -  \frac{a_{np}^6}{(np)^6} \right) \\
=& O\left( a_{np}^6 (np)^{-2} \right).
\end{align*}
The last calculation shows that the expectation in \eqref{crazyexpectation} is equal to $E\left( \prod_{i=1}^3  Z_{1,t_i-k_i}^L Z_{2,t_i-l_i}^L \right)$ plus lower order terms, and that the leading term is of order $a_{np}^6 (np)^{-2}$. With this observation we can finally conclude for \eqref{crazyexpectation} that
\begin{align*}
&\frac{64}{\epsilon^3}\frac{p^5}{a_{np}^6}\sum_{t_1,t_2,t_3=1}^n \sum_{k_1,k_2,k_3}\sum_{l_1,l_2,l_3} \prod_{j=1}^3 (c_{k_j}c_{l_j}) E\left( \prod_{i=1}^3 \left( Z_{1,t_i-k_i}^L Z_{2,t_i-l_i}^L - \xi_n^2 \right) \right) \\
&= O\left( \frac{64}{\epsilon^3}\frac{p^5}{a_{np}^6} n a_{np}^6 (np)^{-2} \right) 
= \frac{64}{\epsilon^3} O\left(\frac{p^3}{n}\right)\to 0, 
\end{align*}
which goes to zero by assumption. This completes the proof for $\alpha\in[2,3)$.\\

The method to deal with $\alpha\in[3,4)$ is similar to the one before and thus only described briefly. We use Markov's inequality with power four to obtain that the term in \eqref{mainterm} is bounded by
\begin{align}
\frac{256}{\epsilon^4}&\frac{p^6}{a_{np}^8}\sum_{t_1,t_2,t_3,t_4=1}^n \sum_{k_1,k_2,k_3,k_4}\sum_{l_1,l_2,l_3,l_4} \prod_{j=1}^4 (c_{k_j}c_{l_j}) E\left( \prod_{i=1}^4 \left( Z_{1,t_i-k_i}^L Z_{2,t_i-l_i}^L - \xi_n^2 \right) \right). \label{crazyexpectation2}
\end{align}
Observe that the expectation in \eqref{crazyexpectation2} is only non-zero if either
\begin{enumerate}
\item all index pairs $\{(t_i-k_i,t_i-l_i)\}_{i=1,2,3,4}$ overlap, or
\item there exist exactly two sets of overlapping index pairs, such that no index pair from one set overlaps with an index pair from the other set. We call these two sets disjoint.
\end{enumerate}
Case (i) is similar to the previous case, so that one can see that
\[ E\left( \prod_{i=1}^4 \left( Z_{1,t_i-k_i}^L Z_{2,t_i-l_i}^L - \xi_n^2 \right) \right) = O\left( (E((Z_{11}^L)^4))^2 \right) = O\left( \frac{a_{np}^8}{(np)^2} \right),\]
and that the contribution of $\sum_{t_1,t_2,t_3,t_4=1}^n$ is of order $n$. Therefore, in this case, the term in \eqref{crazyexpectation2} is of the order
\[ \frac{256}{\epsilon^4} \frac{p^6}{a_{np}^8} O\left( n \frac{a_{np}^8}{(np)^2} \right) = \frac{256}{\epsilon^4} O\left( \frac{p^4}{n} \right)\to 0. \]
Thus, we only have to determine the contribution in case (ii). Since the two sets of overlapping index pairs are disjoint, we obtain that 
\[ E\left( \prod_{i=1}^4 \left( Z_{1,t_i-k_i}^L Z_{2,t_i-l_i}^L - \xi_n^2 \right) \right) =  E\left( \left( Z_{11}^L Z_{21}^L - \xi_n^2 \right)^2 \right)^2. \]
Since $\alpha>2$ this converges to a constant. In contrast to case (i), the contribution of $\sum_{t_1,t_2,t_3,t_4=1}^n$ is of order $n^2$. This is due to the fact that the two sets of overlapping index pairs are disjoint, hence only two out of the four indices $t_1,\ldots,t_4$ are given by linear combinations of the other two and the $k's$ and $l's$. Therefore \eqref{crazyexpectation2} is of the order
\[ \frac{256}{\epsilon^4} \frac{p^6}{a_{np}^8} O\left( n^2  \right)\to 0. \]
The convergence to zero is justified by
\[  \frac{p^6}{a_{np}^8} n^2 = n^{2-8/\alpha} p^{6-8/\alpha} L(np)^{-8} \leq O\left( n^{2-8/\alpha+\beta(6-8/\alpha)} L(n^{\beta+1})^{-8} \right) \to 0,\]
since $\beta<\frac{4-\alpha}{3\alpha-4}$. This completes the proof of \cref{th conv op norm 2}.
\end{proof}

\subsection{Extremes on the diagonal}\label{extremes diagonal}

In this section we analyze the extremes of the diagonal entries of $XX^T$, which are partial sums of squares of linear processes. To this end, we start with two auxiliary results. While \cref{lemma pull c sum out} is only valid for $\alpha<2$, \cref{lemma pp conv} covers the case where $2\leq\alpha<4$. Subsequently, these two lemmas help us to establish a general limit theorem for the diagonal entries of $XX^T$ for $0<\alpha<4$ in \cref{pp conv inf order ma}, the major result of this section.

\begin{lemma}\label{lemma pull c sum out}
Let $(Z_t)$ be an iid sequence such that $nP(|Z_1|>a_n x)\to x^{-\alpha}$ with $\alpha\in(0,2)$. For any sequence $(c_j)$ satisfying \eqref{summability c_j} we have, if $p$ and $n$ go to infinity, that
\[ pP\left( \sum_{t=1}^n \sum_{j=-\infty}^\infty c_j^2 Z_{t-j}^2 > a_{np}^2 x \right) \to \left(\sum_{j=-\infty}^\infty c_j^2 \right)^{\frac{\alpha}{2}} x^{-\alpha/2} \]
\end{lemma}

\begin{proof}
Fix some $x>0$. Observe that \cref{large deviation} and \eqref{a_np} imply for $n\to\infty$ that $pP(\sum_{t=1}^n Z_t^2>a_{np}^2x)\to{}x^{-\alpha/2}$.
 We begin by showing the claim for a linear process of finite order. For any $\eta>0$ we have
\begin{align*}
P\left( \left| \sum_{j=-m}^m c_j^2  \sum_{t=1}^n Z_t^2 - \sum_{t=1}^n \sum_{j=-m}^m c_j^2 Z_{t-j}^2  \right| > a_{np}^2 \eta \right) \leq& P\left(  \sum_{j=-m}^m c_j^2 \sum_{t=1-j}^{j} Z_t^2  > a_{np}^2 \eta \right) \ston{} 0.
\end{align*}
Consequently,
\begin{align}
\lim_{n\to\infty} pP\left( \sum_{t=1}^n \sum_{j=-m}^m c_j^2 Z_{t-j}^2 > a_{np}^2 x \right) = x^{-\alpha/2} \left(\sum_{j=-m}^m c_j^2\right)^{\frac{\alpha}{2}}.
\label{claim for finite ma}
\end{align}
This and the positivity of the summands implies 
\begin{align}
\liminf_{n\to\infty} pP\left( \sum_{t=1}^n \sum_{j=-\infty}^\infty c_j^2 Z_{t-j}^2 > a_{np}^2 x \right) \geq x^{-\alpha/2} \left(\sum_{j=-\infty}^\infty c_j^2 \right)^{\frac{\alpha}{2}}.
\label{liminf lin process}
\end{align}
Thus it is only left to show that the limsup is bounded by the right hand side of \eqref{liminf lin process}. 
Using Markov's inequality yields
\begin{align*}
p P\left( \sum_{t=1}^n \sum_{j=-\infty}^\infty c_j^2 Z_{t-j}^2 > a_{np}^2 x \right) \leq \sum_{j=-\infty}^\infty pn P\left( c_j^2 Z_{1}^2 > {a_{np}^2 x} \right) + \sum_{j=-\infty}^\infty c_j^2 \frac{pn}{a_{np}^2 x} E\left( Z_{1}^2 \1_{\{ c_j^2 Z_{1}^2 \leq a_{np}^2 x \}} \right).
\end{align*}
Since $E\left( Z_{1}^2 \1_{\{ Z_{1}^2 \leq \cdot\}} \right)$ is a regularly varying function with index $\alpha/2-1$ we obtain, by Potter's bound, Karamata's Theorem and \eqref{summability c_j}, that, for some constant $C_1>0$,
\begin{align*}
 c_j^2 \frac{pn}{a_{np}^2 x} E\left( Z_{1}^2 \1_{\{ c_j^2 Z_{1}^2 \leq a_{np}^2 x\}} \right) =&  \frac{c_j^2}{x} \frac{E\left( Z_{1}^2 \1_{\{ c_j^2 Z_{1}^2 \leq a_{np}^2 x\}} \right)}{E\left( Z_{1}^2 \1_{\{ Z_{1}^2 \leq a_{np}^2 x\}} \right)} \frac{pn}{a_{np}^2} E\left( Z_{1}^2 \1_{\{ Z_{1}^2 \leq a_{np}^2 x\}} \right) \\
 \leq& C_1 \frac{c_j^2}{x} \left(c_j^{-2}\right)^{1-\alpha/2+(\alpha/2-\delta/2)} x^{1-\alpha/2} = C_1 x^{-\alpha/2} |c_j|^\delta.
\end{align*}
Likewise, $pnP(a_{np}^{-2} Z_1^2> \cdot)$ is a regularly varying function with index $\alpha/2$, thus we obtain, by the same arguments as before, that
\begin{align*}
 pn P\left( c_j^2 Z_{1}^2 > {a_{np}^2 x} \right)\leq C_2 x^{-\alpha/2} |c_j|^\delta.
\end{align*}
With $C=C_1+C_2$ this therefore implies
\begin{align}
\limsup_{n\to\infty} p P\left( \sum_{t=1}^n \sum_{j=-\infty}^\infty c_j^2 Z_{t-j}^2 > a_{np}^2 x \right) \leq C \sum_{j=-\infty}^\infty |c_j|^\delta x^{-\alpha/2}.
\label{eq almost there}
\end{align}
Hence, by \eqref{claim for finite ma} and \eqref{eq almost there}, we finally have, for some $\epsilon\in(0,1)$, that
\begin{align}
& \limsup_{n\to\infty} pP\left( \sum_{t=1}^n \sum_{j=-\infty}^\infty c_j^2 Z_{t-j}^2 > a_{np}^2 x \right) \leq \limsup_{n\to\infty} pP\left( \sum_{t=1}^n \sum_{j=-m}^m c_j^2 Z_{t-j}^2 > (1-2\epsilon)a_{np}^2 x \right) \nonumber\\
&+ \limsup_{n\to\infty} pP\left( \sum_{t=1}^n \sum_{j=m+1}^\infty c_j^2 Z_{t-j}^2 > \epsilon a_{np}^2 x \right) + \limsup_{n\to\infty} pP\left( \sum_{t=1}^n \sum_{j=-\infty}^{-m-1} c_j^2 Z_{t-j}^2 > \epsilon a_{np}^2 x \right) \nonumber\\
& \leq x^{-\alpha/2} \left( (1-2\epsilon)^{-\alpha/2} \left(\sum_{j=-m}^m c_j^2 \right)^{\frac{\alpha}{2}} + C \epsilon^{-\alpha/2}  \sum_{j=m+1}^\infty |c_j|^\delta  + C \epsilon^{-\alpha/2} \sum_{j=-\infty}^{-m-1}|c_j|^\delta \right). \label{someline}
\end{align}
Assumption \eqref{summability c_j} shows that the last two terms in \eqref{someline} vanish for $m\to\infty$. Letting $\epsilon\to 0$ thereafter completes the proof.
\end{proof}

For $2\leq\alpha<4$ and $m$-dependence, we state the upcoming lemma.


\begin{lemma}\label{lemma pp conv} 
Assume that there exists an $m\in\N$ such that $c_j=0$ if $|j|>m$. Then we have, for $2\leq\alpha<4$ and $p,n$ going to infinity, that 
\begin{align}
 \sum_{i=1}^p \eps_{a_{np}^{-2}\left(\sum_{t=1}^n X_{it}^2-n\mu_{X,\alpha}\right)} \to \sum_{i=1}^\infty \eps_{\Gamma_i^{-2/\alpha} \sum_{j=-m}^m c_j^2}.
\end{align} 
\end{lemma}

\begin{proof}
Note that we replace $\mu_{X,\alpha}$ by $\mu_X$ in the following to simplify the notation.
For any iid sequence $(Z_t)$ with tail index $2<\alpha<4$ we have that
\begin{align}
 p P\left(\sum_{t=1}^n Z_{t}^2-n\mu_Z > a_{np}^2 x\right) \to x^{-\alpha/2}
\label{large deviation iid}
\end{align}
where $\mu_Z=EZ_1^2$. Indeed, \cite{Heyde1968}, and in greater generality also \cite{Davis1995}, show that, for any $x>0$,
  \begin{align}\label{largedev_iid}
  \frac{P\left(\sum_{t=1}^n Z_t^2-n\mu_Z > a_{np}^2 x\right)}{nP\left(Z_1^2-\mu_Z > a_{np}^2 x \right)} \to  1.
 \end{align}
With $P\left(Z_1^2-\mu_Z > a_{np}^2 x \right)\sim P\left(Z_1^2 > a_{np}^2 x \right) \sim p^{-1} x^{-\alpha/2}$, the result follows. Note that \eqref{largedev_iid} also holds for $\alpha=2$ if $EZ_{11}^2<\infty$. In case $EZ_{11}^2=\infty$ (which can only happen if $\alpha=2$), one has to replace $\mu_Z$ by the sequence of truncated means $\mu_z^n=E(Z_{11}^2\1_{\{ Z_{11}^2 \leq a_{np}^2 \}})$. For notational simplicity, we exclude infinite variance case in the following. It is treated analogously to the finite variance case, except that everywhere $\mu_Z$ has to be replaced by $\mu_Z^n$, $\mu_X$ by $\mu_X^n=\sum_k c_k^2 \mu_Z^n$, and finally $\mu_{X,m}$ by  $\mu_{X,m}^n=\sum_{|k|\leq m} c_k^2 \mu_Z^n$.
By the stationarity of the Z's we have that
\begin{align*}
 P&\left(\left|\sum_{j=-m}^m c_j^2 \sum_{t=1}^n (Z_{1,t}^2-\mu_Z) - \sum_{t=1}^n \sum_{j=-m}^m c_j^2  (Z_{1,t-j}^2-\mu_Z) \right| > a_{np}^2 \eta \right) \\  &\leq P\left(\sum_{j=-m}^m c_j^2 \sum_{t=1-j}^j Z_{1,t}^2 > a_{np}^2 \eta \right)\to 0.
\end{align*}
Hence, using \eqref{large deviation iid}, this yields
\begin{align}\label{dep of Z on j}
 p P\left(\sum_{t=1}^n \sum_{j=-m}^m c_j^2 (Z_{1,t-j}^2-\mu_Z) > a_{np}^2 x\right) \to x^{-\alpha/2} \left|\sum_{j=-m}^m c_j^2\right|^{\alpha/2}.
\end{align}
This immediately implies that
\begin{align*}
 \sum_{i=1}^p \eps_{a_{np}^{-2}\left(\sum_{t=1}^n  \sum_{j=-m}^m c_j^2 (Z_{i,t-j}^2 - \mu_Z)\right)} \to \sum_{i=1}^\infty \eps_{\Gamma_i^{-2/\alpha} \left|\sum_{j=-m}^m c_j^2\right|^{\alpha/2}}.
\end{align*}
Thus it is only left to show that, for any continuous $f:\R_+\to\R_+$ with compact support,
\begin{align*}
 \lim_{n\to\infty}P\left(\sum_{i=1}^p\left|f\left(a_{np}^{-2}\Big(\sum_{t=1}^n X_{it}^2-n\mu_X\Big)\right)-f\left(a_{np}^{-2}\sum_{t=1}^n  \sum_{j=-m}^m c_j^2 (Z_{i,t-j}^2 - \mu_Z)\right)\right|>\eta\right)=0.
\end{align*}
For convenience, we define $f(x)=0$ if $x\leq 0$. Clearly, we have that
\begin{align*}
 \left| \sum_{t=1}^n X_{it}^2-n\mu_X - \sum_{t=1}^n  \sum_{j=-m}^m c_j^2 (Z_{i,t-j}^2 - \mu_Z) \right|\leq 2 \sum_{j=-m}^{m-1}\sum_{k=j+1}^m |c_j c_k| \left| \sum_{t=1}^n Z_{i,t-j} Z_{i,t-k} \right|.
\end{align*}
Hence, it suffices to show that
\begin{align*}
 a_{np}^{-2} \max_{1\leq i\leq p}\left| \sum_{t\in J_s} Z_{i,t-j} Z_{i,t-k} \right|\to 0
\end{align*}
for each fixed $j\in\{-m,\ldots,m-1 \}$, $k\in\{j+1,\ldots,m \}$ and $s\in\{0,\ldots,k-j\}$, where $J_s\coloneqq s+(k-j+1)\N_0$. Note that $(Z_{i,t-j} Z_{i,t-k})_{t\in J_s}$ is a sequence of iid random variables with mean zero. Therefore we have, by Markov's inequality,
\begin{align*}
 P\left( \max_{1\leq i\leq p}\left| \sum_{t\in J_s} Z_{i,t-j} Z_{i,t-k} \right| > a_{np}^2 \eta \right) \leq& p P\left( \left| \sum_{t\in J_s} Z_{1,t-j} Z_{1,t-k} \right| > a_{np}^2 \eta \right) \\ \leq& \frac{p}{\eta^2 a_{np}^4} \sum_{t\in J_s} \var(Z_{1,t-j} Z_{1,t-k}) \\ \leq& \frac{pn}{\eta^2 a_{np}^4} (EZ_{11}^2)^2 \\ =& O\left( \frac{pn}{a_{np}^4} \right) = O\left( (pn)^{1-4/\alpha} L(pn)^{-4} \right) \to 0 
\end{align*}
since $\alpha<4$.
\end{proof}

Now we prove the major result of this section, that is, the point process convergence of the diagonal elements of the sample covariance $XX^\T$ (or its centered version). This indirectly characterizes the extremal behavior of the $k$-largest diagonal entries of $XX^\T$. Note that \cref{pp conv inf order ma} holds for any $0<\beta<\infty$ in \eqref{beta con} independently of $\alpha\in(0,4)$.



\begin{proposition}\label{pp conv inf order ma}
Let $0<\alpha<4$ and suppose that \eqref{beta con} holds for some $\beta>0$. Then we have that 
\begin{align}\label{eq lin process conv}
 \sum_{i=1}^p \eps_{a_{np}^{-2}\left(\sum_{t=1}^n X_{it}^2-n\mu_{X,\alpha}\right)} \to \sum_{i=1}^\infty \eps_{\Gamma_i^{-2/\alpha} \sum_{j=-\infty}^\infty c_j^2}
\end{align}
with $\mu_{X,\alpha}$ and $(\Gamma_i)$ as given in \eqref{mu X alpha} and \eqref{N via Gamma}, respectively.
\end{proposition}

\begin{proof}
For notational simplicity we assume without loss of generality that $X_{it}=\sum_{j=0}^\infty c_j Z_{i,t-j}$, and write $\mu_X=\mu_{X,\alpha}$. The extension to the non-causal case is obvious.\\

We begin with the case of $0<\alpha<2$. First we prove the claim for finite linear processes $X_{it,m}=\sum_{j=0}^m c_j Z_{i,t-j}$.
From \cref{lemma pull c sum out} we already have that
\begin{align}
\sum_{i=1}^p  \eps_{a_{np}^{-2}\sum_{t=1}^n \sum_{j=0}^m c_j^2 Z_{i,t-j}^2}  \ston{D} \sum_{i=1}^\infty \eps_{\Gamma_i^{-2/\alpha} \sum_{j=0}^m c_j^2}.
\label{pp without cross products}\end{align}
Thus it is only left to show that all terms involving cross products are negligible. By \cite[Theorem 4.2]{Kallenberg1983} it suffices to show, for any $\eta>0$, that
\begin{align}
\lim_{n\to\infty} P\left( \sum_{i=1}^p \left| f\left(a_{np}^{-2}\sum_{t=1}^n X_{it,m}^2\right) - f\left(a_{np}^{-2}\sum_{t=1}^n\sum_{j=0}^m c_j^2 Z_{i,t-j}^2 \right) \right| > \eta \right) = 0
\label{conv xitm to sum of sq}\end{align}
for any continuous function $f:\R_+\to\R_+$ with compact support $\mathrm{supp}(f)\subset [c,\infty]$ and $c>0$. Choose some $0<\gamma<c$ and let $K=[c-\gamma,\infty]$. On the set
\[ A_n^\gamma = \left\{ \max_{1\leq i\leq p} \left|\sum_{t=1}^n X_{it,m}^2-\sum_{t=1}^n\sum_{j=0}^m c_j^2 Z_{i,t-j}^2\right| \leq a_{np}^2 \gamma \right\} \]
the following is true: if $a_{np}^{-2}\sum_{t=1}^n\sum_{j=0}^m c_j^2 Z_{i,t-j}^2\notin K$, then the absolute difference in \eqref{conv xitm to sum of sq} is zero, else it is bounded by the modulus of continuity $\omega(\gamma)=\sup\{|f(x)-f(y)|: |x-y|\leq\gamma\}$.
Hence, the probability in \eqref{conv xitm to sum of sq} is bounded by
\begin{align*}
P\left( \omega(\gamma) \sum_{i=1}^p \eps_{a_{np}^{-2}\sum_{t=1}^n\sum_{j=0}^m c_j^2 Z_{i,t-j}^2}(K) > \eta \right) + P\left(\left(A_n^\gamma\right)^c\right).
\end{align*}
By \eqref{pp without cross products}, the first summand converges to
\[ P\left( \omega(\gamma) \sum_{i=1}^\infty \eps_{\sum_{j=0}^m c_j^2 \Gamma_i^{-2/\alpha}}(K) > \eta \right). \]
Since $\sum_{i=1}^\infty \eps_{\sum_{j=0}^m c_j^2 \Gamma_i^{-2/\alpha}}(K)<\infty$ and $\omega(\gamma)\to 0$ as $\gamma\to 0$, this probability approaches zero as $\gamma$ tends to zero. To show that
\begin{align}
P\left(\left(A_n^\gamma\right)^c\right) \leq P\left( 2 \sum_{j=0}^{m-1} \sum_{k=j+1}^m |c_j c_k | \max_{i=1:p} \sum_{t=1}^n |Z_{i,t-j} Z_{i,t-k}| > a_{np}^2 \gamma \right) \ston{} 0
\label{An complement}\end{align}
we use the following observation for fixed $j\in\{0,\ldots,m-1\}$ and $k\in\{j+1,\ldots,m\}$: the product $Z_{i,t-j} Z_{i,t-k}$ has, because of independence, tail index $\alpha$, and   $Z_{i,t-j} Z_{i,t-k}$ and $Z_{i,s-j} Z_{i,s-k}$ are independent if and only if $|s-t|\neq k-j$. Thus, we partition the natural numbers $\N$ into $k-j+1$ pairwise disjoint sets $s+(k-j+1)\N_0$, $s\in\{0,\ldots,k-j\}$. Then we have, by \cref{lemma non-diag elemements} and the independence of the summands, that 
\[ a_{np}^{-2} \max_{1\leq i\leq p} \sum_{t\in s+(k-j+1)\N_0} |Z_{i,t-j} Z_{i,t-k}| \ston{P} 0, \]
for each $s\in\{0,\ldots,k-j\}$. Since $j,k$ only vary over finite sets this implies \eqref{An complement}.
Therefore we have shown \eqref{eq lin process conv} for a finite order moving average $X_{it,m}$. 

Now we let $m$ go to infinity. Clearly, we have that
\begin{align}
\sum_{i=1}^\infty \eps_{\Gamma_i^{-2/\alpha} \sum_{j=0}^m c_j^2} \sto{D}{m\to\infty} \sum_{i=1}^\infty \eps_{\Gamma_i^{-2/\alpha} \sum_{j=0}^\infty c_j^2}.
\label{pp conv linear proc}
\end{align}
Thus, by \cite[Theorem 3.2]{Billingsley1999}, it is only left to show that
\begin{align} 
\lim_{m\to\infty} \limsup_{n\to\infty} P\left( \sum_{i=1}^p \left| f\left(a_{np}^{-2}\sum_{t=1}^n X_{it}^2\right) - f\left(a_{np}^{-2}\sum_{t=1}^n X_{it,m}^2\right) \right| > \eta \right) = 0.
\label{eq vague conv f}\end{align}
By repeating the previous arguments, it suffices to show
\[ \limsup_{n\to\infty} P\left( a_{np}^{-2} \max_{1\leq i\leq p} \sum_{t=1}^n |X_{it}^2-X_{it,m}^2| > \gamma\right) \leq \limsup_{n\to\infty} pP\left( a_{np}^{-2} \sum_{t=1}^n |X_{1t}^2-X_{1t,m}^2| > \gamma\right) \to 0, \]
as $m\to\infty$. Clearly, we have that
\begin{align}
X_{1t}^2-X_{1t,m}^2 &= \sum_{j=m+1}^\infty c_j^2 Z_{1,t-j}^2 + 2 \sum_{j=m+1}^\infty \sum_{k=0}^m c_j c_k Z_{1,t-j} Z_{1,t-k} + \sum_{j=m+1}^\infty  \sum_{\substack{k=m+1\\k\neq j}}^\infty c_j c_k Z_{1,t-j} Z_{1,t-k}.
\label{rest}\end{align}
For the first summand on the right hand side of equation \eqref{rest} we have, by \cref{lemma pull c sum out}, that
\begin{align*}
p P\left( \sum_{t=1}^n \sum_{j=m+1}^\infty c_j^2 Z_{1,t-j}^2 > \eta a_{np}^2 \right) \ston{} \left(\sum_{j=m+1}^\infty c_j^2\right)^{\alpha/2} \eta^{-\alpha/2} \sto{}{m\to\infty} 0.
\end{align*}
Using \cref{lemma pull c sum out} and the elementary inequality $2|ab|\leq a^2+b^2$, we obtain for the second term in equation \eqref{rest} that
\begin{align*}
p & P\left( 2 \sum_{t=1}^n \sum_{j=m+1}^\infty \sum_{k=0}^m |c_j c_k Z_{1,t-j} Z_{1,t-k}| > \eta a_{np}^2 \right) \leq p P\left( \sum_{t=1}^n \sum_{j=m+1}^\infty \sum_{k=0}^m |c_j c_k| Z_{1,t-j}^2 > \frac{\eta}{2} a_{np}^2 \right) \\
&+ p P\left( \sum_{t=1}^n \sum_{j=m+1}^\infty \sum_{k=0}^m |c_j c_k| Z_{1,t-k}^2 > \frac{\eta}{2} a_{np}^2 \right) \sim 2\frac{\eta}{4}^{-\alpha/2} \left(\sum_{k=0}^m |c_k|\right)^{\alpha/2} \left(\sum_{j=m+1}^\infty |c_j|\right)^{\alpha/2},
\end{align*}
and since $\sum_{j=0}^\infty |c_j|<\infty$, this term converges to zero as $m\to\infty$. The third term in equation \eqref{rest} can be handled similarly. Thus the proof is complete for $0<\alpha<2$.\\

For $2\leq\alpha<4$, \cref{lemma pp conv} gives us the result for a finite moving average. Thus it is only left to show that
to show that
\[ \lim_{m\to\infty}\limsup_{n\to\infty}P\left( \sum_{i=1}^p \left| f\Big(a_{np}^{-2}\sum_{t=1}^n(X_{it}^2-\mu_X)\Big) - f\Big(a_{np}^{-2}\sum_{t=1}^n(X_{it,m}^2-\mu_{X,m})\Big) \right| > \gamma \right)=0 \]
for any continuous $f$ with compact support and $\gamma>0$. By the arguments given before it suffices to show that
\begin{align*}
\lim_{m\to\infty}\limsup_{n\to\infty}pP\left( \left| \sum_{t=1}^n (X_{1t}^2-X_{1t,m}^2-(\mu_X-\mu_{X,m})) \right| > a_{np}^2\gamma\right)=0.
\end{align*}
Clearly, we have that
\begin{align*}
pP&\left( \left| \sum_{t=1}^n (X_{1t}^2-X_{1t,m}^2-(\mu_X-\mu_{X,m})) \right| > a_{np}^2\gamma\right) \\ 
\leq& pP\left( \left| \sum_{k=m+1}^\infty c_k^2 \sum_{t=1}^n (Z_{1,t-k}^2-\mu_Z) \right| > a_{np}^2\frac{\gamma}{3}\right) \\
&+ pP\left( 2 \left| \sum_{k=m+1}^\infty \sum_{l=0}^m c_k c_l \sum_{t=1}^n Z_{1,t-k}Z_{1,t-l} \right| > a_{np}^2\frac{\gamma}{3}\right) \\
&+ pP\left( 2 \left| \sum_{k=m+1}^\infty \sum_{l=k+1}^\infty c_k c_l \sum_{t=1}^n Z_{1,t-k}Z_{1,t-l} \right| > a_{np}^2\frac{\gamma}{3}\right) \\
&= \mrm{I} + \mrm{II} + \mrm{III}
\end{align*}
We will show in turn that $\mrm{I},\mrm{II},\mrm{III}\to 0$. We begin with $\mrm{I}$. Clearly, there either exist a $t$ and a $k$ such that $|c_k Z_{1,t-k}>a_{np}|$, or $|c_k Z_{1,t-k}\leq a_{np}|$ for all $t,k$. This simple fact and Chebyshev's inequality yield
\begin{align*}
\mrm{I} =& pP\left( \left| \sum_{k=m+1}^\infty c_k^2 \sum_{t=1}^n (Z_{1,t-k}^2-\mu_Z) \right| > a_{np}^2\frac{\gamma}{3}\right) \\
\leq& \sum_{k=m+1}^\infty pn P(|c_k Z_{1,1-k}>a_{np}|) + \frac{3}{\gamma} \frac{p}{a_{np}^4} \var\left(  \sum_{k=m+1}^\infty c_k^2 \sum_{t=1}^n Z_{1,t-k}^2 \1_{\{ |c_k Z_{1,t-k}|\leq a_{np} \}} \right) \\
&+ p \1_{\left\{ \sum_{k=m+1}^\infty c_k^2 n E\left(Z_{11}^2 \1_{\{ |c_k Z_{1,t-k}|> a_{np} \}}\right) > a_{np}^2\frac{\gamma}{3}  \right\}} = \mrm{I_1}+\mrm{I_2}+\mrm{I_3}
\end{align*}
For the first term we have by Karamata's theorem that 
\[ \lim_{m\to\infty}\limsup_{n\to\infty}\mrm{I_1}=\lim_{m\to\infty} \sum_{k=m+1}^\infty c_k^\alpha = 0.\]
Another application of Karamata's theorem shows that
\[ E\left(Z_{11}^2 \1_{\{ |c_k Z_{1,t-k}|> a_{np} \}}\right) \sim |c_k|^{\alpha/2-1} \frac{a_{np}^2}{np}, \]
therefore
\[ \lim_{n\to\infty}\frac{p \1_{\left\{ \sum_{k=m+1}^\infty c_k^{\alpha/2+1}  > p \frac{\gamma}{3}  \right\}}}{\mrm{I_3}}=1. \]
However, $p \1_{\left\{ \sum_{k=m+1}^\infty c_k^{\alpha/2+1}  > p \frac{\gamma}{3}  \right\}}=0$ for $n$ sufficiently large, since $p=p_n\to\infty$ and 
\[ \sum_{k=m+1}^\infty c_k^{\alpha/2+1}<\infty.\]
As a consequence, $\mrm{I_3}\to 0$. Regarding $\mrm{I_2}$, observe that the covariance in
\begin{align*}
\mrm{I_2}=\frac{3}{\gamma} \frac{p}{a_{np}^4} \sum_{k=m+1}^\infty \sum_{k'=m+1}^\infty c_k^2 c_{k'}^2 \sum_{t=1}^n \sum_{t'=1}^n \cov\left(   Z_{1,t-k}^2 \1_{\{ |c_k Z_{1,t-k}|\leq a_{np} \}},Z_{1,t'-k'}^2 \1_{\{ |c_k' Z_{1,t'-k'}|\leq a_{np} \}} \right)
\end{align*}
is zero if $t-k\neq t'-k'$. In the case of equality, $t-k= t'-k'$, we have that
\begin{align*} 
\sum_{t=1}^n &\sum_{t'=1}^n \cov\left(  Z_{1,t-k}^2 \1_{\{ |c_k Z_{1,t-k}|\leq a_{np} \}},Z_{1,t'-k'}^2 \1_{\{ |c_k' Z_{1,t'-k'}|\leq a_{np} \}} \right) \\
=& \sum_{t=1}^n  \var\left(  Z_{1,t-k}^2 \1_{\{ |c_k Z_{1,t-k}|\leq a_{np} \}} \1_{\{ |c_k' Z_{1,t'-k'}|\leq a_{np} \}} \right) \\
\leq& n E\left( Z_{1,1-k}^4 \1_{\{ |\min\{c_k,c_{k'}\} Z_{1,1-k}|\leq a_{np} \}}  \right)
\end{align*}
Using Karamata's theorem and Potter's bound we obtain that there exists a $C>0$ and an $\epsilon>0$ such that
\[ \frac{pn}{a_{np}^4} E\left( Z_{1,1-k}^4 \1_{\{ |\min\{c_k,c_{k'}\} Z_{1,1-k}|\leq a_{np} \}}  \right) \leq C \min\{c_k,c_{k'}\}^{\alpha/4-\epsilon-1}. \]
For $m$ sufficiently large the coefficients become smaller than one, thus
\[ \min\{c_k,c_{k'}\}^{\alpha/4-\epsilon-1} \leq c_k^{\alpha/4-\epsilon-1} c_{k'}^{\alpha/4-\epsilon-1}. \]
All in all we obtain
\[  \lim_{m\to\infty}\limsup_{n\to\infty}\mrm{I_2}\leq \frac{3C}{\gamma} \lim_{m\to\infty} \left( \sum_{k=m+1}^\infty c_k^{1+\alpha/4-\epsilon} \right)^2=0, \]
since $ \sum_{k=0}^\infty c_k<\infty$. For the second term observe that it follows, using Chebyshev's inequality, $EZ_{11}=0$ and the independence of the $Z$'s, that
\begin{align*}
\mrm{II} =& pP\left( 2 \left| \sum_{k=m+1}^\infty \sum_{l=0}^m c_k c_l \sum_{t=1}^n Z_{1,t-k}Z_{1,t-l} \right| > a_{np}^2\frac{\gamma}{3}\right) \\
\leq& \frac{6}{\gamma} \frac{p}{a_{np}^4}\var\left( \sum_{k=m+1}^\infty \sum_{l=0}^m c_k c_l \sum_{t=1}^n Z_{1,t-k}Z_{1,t-l} \right) \\
=& \frac{6}{\gamma} \frac{p}{a_{np}^4} \sum_{k,k'=m+1}^\infty \sum_{l,l'=0}^m c_k c_{k'} c_l c_{l'} \sum_{t,t'=1}^n E\left(  Z_{1,t-k}Z_{1,t-l}Z_{1,t'-k'}Z_{1,t'-l'} \right) \\
\leq& \frac{6}{\gamma} \frac{p}{a_{np}^4} \sum_{k,k'=m+1}^\infty \sum_{l,l'=0}^m c_k c_{k'} c_l c_{l'} n E\left(  Z_{11}^2 \right)^2 \\
\leq& O\left(  \left(\sum_{k=m+1}^\infty c_k\right)^2 \left(\sum_{l=0}^m c_l\right)^2 \frac{pn}{a_{np}^4} \right) \ston{} 0,
\end{align*}
since $2<\alpha<4$. The remaining term $\mrm{III}$ can be dealt with similarly to the previous term $\mrm{II}$. Hence the proof is complete.
\end{proof}

\subsection{Proofs of \cref{th1} and \cref{th1b}}\label{proof th1}

In this section we use the foregoing results from \cref{sec large dev}, \cref{opnorm} and \cref{extremes diagonal} to complete the proofs of \cref{th1}  and \cref{th1b}. 

\begin{proof}[Proof of \cref{th1}]
Denote by $S_k=(XX^\T)_{kk}=\sum_{t=1}^n X_{kt}^2$ the diagonal entries of $XX^\T$. 
Recall that $\lambda_{(1)}\geq\ldots\geq\lambda_{(p)}$ are the upper order statistics of the eigenvalues of $XX^\T-n\mu_{X,\alpha}I_p$ with $\mu_{X,\alpha}$ as given in \eqref{mu X alpha}. Similarly we denote by $S_{(1)}\geq\ldots\geq S_{(p)}$ the upper order statistics of $S_k-n\mu_{X,\alpha}=\sum_{t=1}^n X_{kt}^2-n\mu_{X,\alpha}$. Weyl's Inequality, cf. \cite[Corollary III.2.6]{Bhatia1997}, and \Cref{th conv op norm 2} imply that
\begin{align}
a_{np}^{-2} \max_{1\leq k\leq p}|\lambda_{(k)}-S_{(k)}|= a_{np}^{-2} \max_{1\leq k\leq p}|\lambda_{k}-S_{k}|\leq a_{np}^{-2} \norm{XX^\T-D}_2 \ston{P} 0,
\label{Weyl}\end{align}
where $D=\diag(XX^\T)$. From \cref{pp conv inf order ma} we have
\begin{align}
\wh N_n = \sum_{i=1}^p  \eps_{a_{np}^{-2} (S_{i}-n\mu_{X,\alpha})} = \sum_{i=1}^p  \eps_{a_{np}^{-2} S_{(i)}} \ston{D} N. 
\label{pp conv max}\end{align}
Thus, by \cite[Theorem 4.2]{Kallenberg1983}, it suffices to show that
\[ P(|\wh N_n(f)-N_n(f)|>\eta) \leq P\left(\sum_{i=1}^p{\left|f\left(\frac{S_{(i)}}{a_{np}^2}\right)-f\left(\frac{\lambda_{(i)}}{a_{np}^2}\right)\right|}>\eta\right)\ston{} 0\]
for a nonnegative continuous function $f$ with compact support $\mathrm{supp}(f)\subset[c,\infty]$, for some $c>0$.
For convenience we set $f(x)=0$ if $x\leq 0$.
 Since $N((c/2,\infty])<\infty$ almost surely, we can choose some $i\in\N$ large enough such that the probability  $P(N((c/2,\infty])\geq i)<\delta/2$. By \eqref{pp conv max}, it follows that
$ P(a_{np}^{-2}S_{(i)}>c/2)=P(\wh N_n((c/2,\infty])\geq i)\to P(N((c/2,\infty])\geq i)$ and thus, for $n$ large enough, $P(a_{np}^{-2}S_{(i)}>c/2)<\delta$. Consequently, by \eqref{Weyl}, it follows that $P(a_{np}^{-2}\lambda_{(i)}\geq c)<2\delta$. Since $a_{np}^{-2}S_{(i)}\leq c/2$ and $a_{np}^{-2}\lambda_{(i)}< c$ imply that both  $f(a_{np}^{-2}M_{(k)})=0$ and $f(a_{np}^{-2}\lambda_{(k)})=0$ for all $k\geq i$, we obtain
\begin{align*}
P&\left(\sum_{j=1}^p{\left|f\left(\frac{S_{(j)}}{a_{np}^2}\right)-f\left(\frac{\lambda_{(j)}}{a_{np}^2}\right)\right|}>\eta\right) \leq  P\left(\sum_{j=1}^p{\left|f\left(\frac{S_{(j)}}{a_{np}^2}\right)-f\left(\frac{\lambda_{(j)}}{a_{np}^2}\right)\right|}>\eta,a_{np}^{-2}S_{(i)}>\frac{c}{2}\right) \nonumber\\
&+ P\left(\sum_{j=1}^p{\left|f\left(\frac{S_{(j)}}{a_{np}^2}\right)-f\left(\frac{\lambda_{(j)}}{a_{np}^2}\right)\right|}>\eta,a_{np}^{-2}S_{(i)}\leq\frac{c}{2},a_{np}^{-2}\lambda_{(i)}\geq c\right) \nonumber\\
&+ P\left(\sum_{j=1}^p{\left|f\left(\frac{S_{(j)}}{a_{np}^2}\right)-f\left(\frac{\lambda_{(j)}}{a_{np}^2}\right)\right|}>\eta,a_{np}^{-2}S_{(i)}\leq\frac{c}{2},a_{np}^{-2}\lambda_{(i)} < c\right) \nonumber\\
\leq& 3\delta + P\left(\sum_{j=1}^{i-1}{\left|f\left(\frac{S_{(j)}}{a_{np}^2}\right)-f\left(\frac{\lambda_{(j)}}{a_{np}^2}\right)\right|}>\eta\right),
\end{align*}
which becomes arbitrarily small due to equation \eqref{Weyl} and the fact that $f$ is uniformly continuous.
\end{proof}

In the case when the entries of $X$ are iid and have tail index $\alpha<2$, we can refine our techniques to weaken the assumptions on the growth of $p=p_n$, cf. \cref{th1b}.

\begin{proof}[Proof of \cref{th1b}]
By assumption $X=(Z_{it})$. First we consider the case (a) and assume that $\kappa\geq 1$. 
We will show that, for any fixed positive integer $k$,
\begin{align}
\frac{\lambda_{(k)}}{S_{(k)}} \ston{P} 1.
\label{claim iid}\end{align}
\Cref{claim iid,pp conv max and sum sq} then imply
\begin{align*}
\left|\frac{S_{(k)}}{a_{np}^2}-\frac{\lambda_{(k)}}{a_{np}^2}\right| = \left|1-\frac{\lambda_{(k)}}{S_{(k)}}\right|\frac{S_{(k)}}{a_{np}^2} \ston{P} 0,
\end{align*}
and hence $N_n\to N$ as in the proof of \cref{th1} (i). Define $M_i=\max_{1\leq t\leq n} X_{it}^2$ and denote by $M_{(1)}\geq\ldots\geq M_{(p)}$ the upper order statistics of $M_1,\ldots,M_p$. Observe that the continuous mapping theorem applied to \eqref{pp conv max and sum sq} and \eqref{pp conv max and sum abs} yields, for any fixed $k$,
\begin{align*}
\frac{S_{(k)}}{M_{(k)}} \ston{P} 1, \quad\tn{ and }\quad \frac{\norm{X}_\infty^2}{M_{(1)}} \ston{P} 1,
\end{align*}
because $\kappa\geq 1$. Now we start showing \eqref{claim iid} by induction. For $k=1$ we have, on the one hand, that
\begin{align*}
\frac{\lambda_{(1)}}{S_{(1)}} = \frac{\norm{X_n X_n^\T}_2}{S_{(1)}} \leq \frac{\norm{X_n}_2^2}{S_{(1)}} \leq \frac{\norm{X_n}_\infty^2}{S_{(1)}} = \frac{\norm{X_n}_\infty^2}{M_{(1)}} \frac{M_{(1)}}{S_{(1)}} \ston{P} 1.
\end{align*}
Let us denote by $e_1,\ldots,e_p$ the standard Euclidean orthonormal basis in $\R^p$ and by $i_1$ the (random) index that satisfies $S_{i_1}=S_{(1)}$. Then we have, on the other hand, by the Minimax Principle \cite[Corollary III.1.2]{Bhatia1997}, that
\begin{align*}
\frac{\lambda_{(1)}}{S_{(1)}} = \frac{\max_{v\in\R^p} \sp{v}{XX^\T v}}{S_{(1)}} \geq \frac{\sp{e_{i_1}}{XX^\T e_{i_1}}}{S_{(1)}} = \frac{S_{i_1}}{S_{(1)}} = 1.
\end{align*}
This shows \eqref{claim iid} for $k=1$. To keep the notation simple, we describe the induction step only for $k=2$. The arguments for the general case are exactly the same. Denote by $i_2$ the random index such that $S_{i_2}=S_{(2)}$. Let $X^{(2)}$ be the $(p-1)\times n$ matrix which is obtained from removing row $i_1$ from $X_n$ and denote by $\varrho_{(1)}$ the largest eigenvalue of $X^{(2)} (X^{(2)})^\T$. Since we have already shown the claim for the largest eigenvalue, it follows that $\varrho_{(1)}/S_{(2)}\to 1$ in probability. By the Cauchy Interlacing Theorem \cite[Corollary III.1.5]{Bhatia1997} this implies $\lambda_{(2)}/S_{(2)}\leq\varrho_{(1)}/S_{(2)}\to 1$. Another application of the Minimax Principle yields
\begin{align*}
\lambda_{(2)} = \max_{\substack{\mc{M}\subset\R^p\\ \mrm{dim}(\mc{M})=2}} \min_{\substack{v\in\mc{M} \\ \norm{v}=1}} v^\T X X^\T v \geq& \min_{\substack{v\in\mrm{span}\{e_{i_1},e_{i_2}\} \\ \norm{v}=1}} v^\T X X^\T v \\
=& \min_{\mu_1,\mu_2\in\R} (\mu_1^2+\mu_2^2)^{-1}  \left( \mu_1^2 S_{(1)} + \mu_2^2 S_{(2)} + 2\mu_1 \mu_2 (X X^\T)_{i_1 i_2} \right).
\end{align*}
Since, by \cref{lemma non-diag elemements} and equation \eqref{pp conv max and sum sq},
\[ \frac{\left| \frac{2\mu_1 \mu_2}{\mu_1^2+\mu_2^2} (X X^\T)_{i_1 i_2} \right|}{S_{(2)}} \leq \frac{a_{np}^{-2} \max_{1\leq i<j\leq p} \sum_{t=1}^n |Z_{it} Z_{jt}|}{a_{np}^{-2} S_{(2)}}  \ston{P} 0. \]
uniformly in $\mu_1,\mu_2\in\R$, an application of the the continuous mapping theorem finally yields that $\lambda_{(2)}/{S_{(2)}}\geq 1+o_P(1)$, where $o_P(1)\to 0$ in probability as $n\to\infty$. Thus the proof for $\kappa\geq 1$ is complete. Now let $\kappa\in(0,1)$. Since $X^\T X$ and $XX^\T$ have the same non-trivial eigenvalues, we consider the transpose $X^\T$ of $X$. This inverts the roles of $p$ and $n$. Therefore, using Potter's bounds and $1/\kappa>1$, the result follows from the same arguments as before. Note that we are in a special case of \cref{th1} (i) if $\kappa=0$. In case (b) we have that $n\sim (1/c\log(p/C))^{1/\kappa}$ is a slowly varying function in $p$, thus an application of \cref{th1} (ii) (a) to $X^\T$ gives the result.
\end{proof}

\subsection{Proof of \cref{th2}}\label{proof extension}

As we shall see, the proof of \cref{th2} will more or less follow the same lines of argument as given for \cref{th1}. We focus on the setting of \cref{th2} (i) here and mention (ii) and (iii) later. The next result is a generalization of \Cref{pp conv inf order ma} allowing for random coefficients.

\begin{proposition}\label{pp conv inf order ma random coeff}
Define $X=(X_{it})$ with $X_{it}$ satisfying \eqref{c_j bounded} and \eqref{def Xit gen}. Suppose \eqref{beta con} holds for some $\beta>0$. If $(\theta_i)$ is a stationary ergodic sequence, then, conditionally on $(\theta_i)$ as well as unconditionally, we have
\begin{align}
\sum_{i=1}^p  \eps_{a_{np}^{-2}(\sum_{t=1}^n X_{it}^2-n\mu_{X,\alpha})} \ston{D} \sum_{i=1}^\infty \eps_{\Gamma_i^{-2/\alpha}\left(E\left|\sum_{j=-\infty}^\infty c_j^2(\theta_1)\right|^{\alpha/2}\right)^{2/\alpha}}
\label{claim random coeff}\end{align}
with $\mu_{X,\alpha}$ and $(\Gamma_i)$ as given in \eqref{mu X alpha} and \eqref{N via Gamma}.
\end{proposition}

\begin{proof}
We prove the cases $0<\alpha<2$ and $2\leq\alpha<4$ separately.\\

Let $0<\alpha<2$. We first prove that, conditionally on $(\theta_i)$,
\begin{align}
  \sum_{i=1}^p  \eps_{a_{np}^{-2}\sum_{t=1}^n \sum_j c_j^2(\theta_i) Z_{i,t-j}^2} \ston{D} \sum_{i=1}^\infty \eps_{\Gamma_i^{-2/\alpha} \left(E\left|\sum_{j} c_j^2(\theta_1)\right|^{\alpha/2}\right)^{2/\alpha} }
\label{claim_ppconv}\end{align}
by showing a.s. convergence of the Laplace functionals. By arguments from the proof of \cite[Proposition 3.17]{Resnick2008} it suffices to show \eqref{cond pp convergence} only for a countable subset of the space of all nonnegative continuous functions with compact support. Thus we fix one nonnegative continuous function $f$ with compact support $\mrm{supp}(f)\subset [c,\infty]$, $c>0$. Conditionally on the process $(\theta_m)$, the points of the point process are independent, and thus
\begin{align}
E\left(e^{-\sum_{i=1}^p f(a_{np}^{-2}\sum_{t=1}^n \sum_j c_j^2(\theta_i) Z_{i,t-j}^2)} \big| (\theta_m) \right) 
&= \prod_{i=1}^p \left(1-\frac{1}{p}\int{(1-e^{-f(x)})pP\left(a_{np}^{-2}\sum_{t=1}^n \sum_j c_j^2(\theta_i) Z_{1,t-j}^2\in dx\Big|\theta_i\right)}\right) \nonumber\\
&= \prod_{i=1}^p \left(1-\frac{1}{p}B_{i,p}\right), \label{ext prod}
\end{align}
where $B_{i,p}=\int{(1-e^{-f(x)})pP(a_{np}^{-2}\sum_{t=1}^n \sum_j c_j^2(\theta_i) Z_{1,t-j}^2\in dx|\theta_i)}$. 
First assume
\begin{align}
\frac{1}{p} \sum_{i=1}^p B_{i,p} \ston{a.s.} B\coloneqq\int{(1-e^{-f(x)})\nu(dx)} 
\label{avg of int}\end{align}
with $\nu$ given by $\nu((x,\infty])\coloneqq x^{-\alpha/2} E\left|\sum_{j} c_j^2(\theta_1)\right|^{\alpha/2}$, 
and
\begin{align}
\frac{1}{p^2} \sum_{i=1}^p B_{i,p}^2 \ston{a.s.} 0.
\label{avg of int sq}\end{align}
Both claims will be justified later. By assumption \eqref{c_j bounded}, we have, using \cref{lemma pull c sum out}, almost surely
\begin{align*}
 B_{i,p}\leq pP\left(a_{np}^{-2}\sum_{t=1}^n \sum_j \tilde{c}_j^2 Z_{1,t-j}^2 > c\right) \ston{} c^{-\alpha/2} \left|\sum_j \tilde{c}_j^2\right|^{\alpha/2},
\end{align*}
and hence there exists a $C>0$ such that $B_{i,p}\leq C$ for all $i,p\in\N$ a.s. The elementary inequality $e^{\frac{-x}{1-x}}\leq 1-x\leq e^{-x} \ \forall x\in[0,1]$, equivalently $e^{\frac{-x^2}{1-x}}\leq {(1-x)}{e^{x}}\leq 1\ \forall x\in[0,1]$, implies together with \eqref{avg of int sq}, for some $c_1>0$, that
\begin{align*}
1\geq \prod_{i=1}^p{{\left(1-\frac{B_{i,p}}{p}\right)}{e^{\frac{B_{i,p}}{p}}}} \geq 
\prod_{i=1}^p{e^{-\frac{B_{i,p}^2}{p^2-pB_{i,p}}}} \geq \prod_{i=1}^p{e^{-\frac{B_{i,p}^2}{p^2-pC}}} \geq 
e^{\frac{-c_1}{p^2}\sum_{i=1}^p{B_{i,p}^2}} \ston{a.s.} 1.
\end{align*}
As a consequence we have that the product in \eqref{ext prod} is asymptotically equivalent to
\begin{align*}
\prod_{i=1}^p e^{-\frac{1}{p}B_{i,p}} = e^{-\frac{1}{p} \sum_{i=1}^p B_{i,p}} \ston{a.s.} e^{-B} = e^{-\int{(1-e^{-f(x)})\nu(dx)}},
\end{align*}
where the convergence follows from \eqref{avg of int}. This implies the almost sure convergence of the conditional Laplace functionals, therefore \eqref{claim_ppconv} holds conditionally on $(\theta_i)$. Using \eqref{c_j bounded} one shows similarly as in the proof of \cref{pp conv inf order ma}, conditionally on $(\theta_i)$, that \eqref{claim_ppconv} implies \eqref{claim random coeff}. Taking the expectation yields that \eqref{claim random coeff} also holds unconditionally.

\textit{Proof of \eqref{avg of int} and \eqref{avg of int sq}.}  As a function in $x$, $pP(\sum_{t=1}^n Z_{1t}^2 > a_{np}^2 x)$ is decreasing and converges pointwise to the continuous function $x^{-\alpha/2}$ as $n\to\infty$. Therefore this convergence is uniform on compact intervals of the form $[x_0,\infty]$ with $x_0>0$. Now fix $x>0$ and let $d_i=\sum_j c_j^2(\theta_i)$. Since $d_i\leq d=\sum_j \tilde{c}_j^2<\infty$ for all $i\in\N$, $\frac{x}{d_i}\geq\frac{x}{d}>0$ is bounded from below, and thus
\begin{align}
 \sup_{i\in\N} \left| pP\left(\sum_{t=1}^n Z_{1t}^2 > a_{np}^2 \frac{x}{d_i} \Big| d_i\right) - x^{-\alpha/2} d_i^{\alpha/2}  \right| \ston{a.s.} 0.
\label{convergence with d_i}\end{align}
Since $(d_i)$ is an instantaneous function of the ergodic sequence $(\theta_i)$, it is also ergodic and thus
\begin{align}
\frac{1}{p}\sum_{i=1}^p{d_i^{\alpha/2} \ston{a.s.} E|d_1|^{\alpha/2}}.
\label{ergodic property}\end{align}
As a consequence of \eqref{convergence with d_i} and \eqref{ergodic property} we obtain
\begin{align*}
&\left| \frac{1}{p}\sum_{i=1}^p{pP\left(\sum_{t=1}^n Z_{1t}^2 > a_{np}^2 \frac{x}{d_i} \Big| d_i\right)} - x^{-\alpha/2} E|d_1|^{\alpha/2} \right| \ston{a.s.} 0.
\end{align*}
Then it is straightforward to show, as in the proof of \cref{lemma pull c sum out}, using \eqref{c_j bounded}, that
\begin{align*}
\frac{1}{p}\sum_{i=1}^p{pP\left(\sum_{t=1}^n \sum_j c_j(\theta_i) Z_{1,t-j}^2 > a_{np}^2 {x} \Big| \theta_i\right)} \ston{a.s.}  x^{-\alpha/2} E|d_1|^{\alpha/2}.
\end{align*}
The vague convergence of above sequence of measures implies $p^{-1} \sum_{i=1}^p B_{i,p} \to B$ almost surely. In exactly the same way one can show that $p^{-1} \sum_{i=1}^p B_{i,p}^2$ converges, thus $p^{-2} \sum_{i=1}^p B_{i,p}^2 \to 0$ a.s., which establishes \eqref{avg of int} and \eqref{avg of int sq} as claimed.\\

Let $2\leq\alpha<4$. 
As before one can show, for any $m<\infty$, that
\begin{align*} \frac{1}{p}\sum_{i=1}^p{pP\left(\sum_{t=1}^n \sum_{j=-m}^m c_j(\theta_i) (Z_{i,t-j}^2-\mu_Z)> a_{np}^2 {x} \Big| \theta_i\right)} \ston{a.s.}  x^{-\alpha/2} E|d_1^m|^{\alpha/2},
\end{align*}
where $d_1^m=\sum_{j=-m}^m c_j^2(\theta_1)$. Hence, an adaptation of the proof of \cref{lemma pp conv} yields, for the truncated process
\[ X_{it,m}=\sum_{k=-m}^m c_k Z_{i,t-k}, \quad \mu_{X,m}=EX_{11,m}^2=\sum_{j=-m}^m c_j^2 \mu_Z, \]
that, conditionally on the sequence $(\theta_i)$,
 \begin{align}\label{pp conv random coef}
 \sum_{i=1}^p \eps_{a_{np}^{-2}\left(\sum_{t=1}^n X_{it,m}^2-n\mu_{X,m}\right)} \to \sum_{i=1}^\infty \eps_{\Gamma_i^{-2/\alpha} \left(E\left|\sum_{j=-m}^m c_j^2(\theta_1)\right|^{\alpha/2}\right)^{2/\alpha}}.
\end{align} 
It is only left to show that this result extends to the more general setting where $m=\infty$. By \cref{pp conv inf order ma} it suffices to show that
\begin{align*}
\lim_{m\to\infty}\limsup_{n\to\infty}\sum_{i=1}^p P\left( \left| \sum_{t=1}^n (X_{it}^2-X_{it,m}^2-(\mu_X-\mu_{X,m})) \right| > a_{np}^2\gamma\Big|(\theta_r)\right)=0.
\end{align*}
To proof this claim, follow the string of arguments of \cref{pp conv inf order ma} and make use of the fact that
\[ \left| \sum_{i=1}^p c_j(\theta_i) \right| \leq p \tilde{c}_j. \]
\end{proof}

\begin{proof}[Proof of \cref{th2}]
\textit{Proof of (i).} If we condition on $(\theta_i)$, the proofs of \cref{th conv op norm,th conv op norm 2} easily carry over to this more general setting when we make use of assumption \eqref{c_j bounded}. Taking the expectation then yields convergence in operator norm unconditionally. A combination of this together with \cref{pp conv inf order ma random coeff} completes the proof.

\textit{Proof of (ii).} Note that \eqref{ergodic property} is the only step in the proof of \Cref{pp conv inf order ma random coeff} where we use the ergodicity of the sequence $(\theta_i)$. But also if $(\theta_i)$ is just stationary, the ergodic theorem implies that the average in \eqref{ergodic property} converges to the random variable $Y=E\left(|d_1|^{\alpha/2}|\mathcal{G}\right)$, where $\mathcal{G}$ is the invariant $\sigma$-field generated by $(\theta_i)$. By construction, $Y$ depends on $\alpha$ and $c_j(\cdot)$, but it is independent of $(\Gamma_i)$, since $(\theta_i)$ is independent of $(Z_{it})$.

\textit{Proof of (iii).} In this setting $(\theta_i)$ is a Markov chain which may not be stationary. But since we derive all results in the proof of \cref{th2} (i) conditionally on $(\theta_i)$ and then take the expectation, stationarity is in fact not needed. 
The theory on Markov chains, see \cite{Meyn2009}, in particular their Theorem 17.1.7 for Markov chains on uncountable state spaces, shows that \eqref{ergodic property} holds if the expectation is taken with respect to the stationary distribution of the Markov chain.
\end{proof}

\section*{Acknowledgements}
R.D. was supported in part by NSF Grant DMS-1107031 and the Institute for Advanced Study at Technische Universit\"at M\"unchen. O.P thanks the Technische Universität München - Institute for Advanced Study, funded by the German Excellence Initiative, and the International Graduate School of Science and Engineering for financial support. R.S was partially supported by the Technische Universität München - Institute for Advanced Study, funded by the German Excellence Initiative

\bibliographystyle{spmpsci}

\end{document}